\documentclass{amsproc}

\usepackage{enumerate, graphicx, xcolor}

\newtheorem{theorem}{Theorem}[section]

\newtheorem{prop}[theorem]{Proposition}
\newtheorem{conj}[theorem]{Conjecture}

\theoremstyle{definition}
\newtheorem{definition}[theorem]{Definition}

\theoremstyle{remark}

\numberwithin{equation}{section}

\newcommand{\la}{\lambda}
\newcommand{\R}{\mathbb{R}}
\definecolor{forest}{rgb}{0.03, 0.47, 0.19}

\usepackage[font=footnotesize]{subcaption}
\begin{document}

\title{Structure of the chromatic polynomial}


\author{Radmila Sazdanovic}
\address{Department of Mathematics, North Carolina State University}
\email{rsazdanovic@math.ncsu.edu}
\thanks{RS was partially supported by NSF grant DMS-1854705.}

\author{Daniel Scofield}
\address{Department of Mathematics, Francis Marion University}
\email{daniel.scofield@fmarion.edu}


\date{\today}

\begin{abstract} 
Recently, big data techniques such as machine learning and topological data analysis have made their way to theoretical mathematics.  Motivated by the recent work with polynomial invariants for knots, we use manifold learning and topological data analysis techniques to explore the structure and properties of the point cloud consisting of the chromatic polynomials of graphs up to 10 crossings. Although chromatic, as well as the Tutte polynomial fail to distinguish graphs, according to a conjecture by Bollob{\'a}s, Pebody and Riordan they approximate the space of random graphs. In this work we compare structures in the chromatic data revealed using filtered PCA and Ball Mapper techniques, and relate them with a range of numerical invariants for graphs.  
\end{abstract}

\maketitle

\section{Introduction}

Graph invariants are properties of graphs that are invariant under graph isomorphism. Their definition implies they can be used to distinguish graphs: if values of the invariant are different the graphs are non-isomorphic, but it is possible to have non-isomorphic graphs with the same invariant. Invariants are abundant in graph theory, as well as in other areas of mathematics such as low-dimensional topology. They may be numerical, polynomials, or more involved algebraic constructions such as homology theories. 
Examples of such invariants for graphs include numbers of vertices, edges, and cycles, connectivity, vertex degrees, and chromatic number, the chromatic, Tutte and magnitude polynomials, the Stanley chromatic function \cite{whitney1932logical, tutte1954contribution, leinster2013magnitude, stanley1995symmetric}, and the corresponding homology theories categorifying these polynomial invariants \cite{helme2005categorification,eastwood2007euler, hepworth2017categorifying, sazdanovic2018categorification}. 

Recent years have brought significant advances of artificial intelligence and machine learning methods in graph theory, which have proven challenging partially due to the heterogeneity of graphs and the wide range of their applications. In this paper we take a different approach motivated by the conjecture of Bollob{\'a}s, Pebody and Riordan stating that almost every pair of independently chosen random graphs are distinguished by their Tutte and chromatic polynomials. Hence, instead of studying graphs directly, we examine the data derived from their chromatic polynomials. This idea of using invariants instead of objects has successfully been implemented in knot theory \cite{jejjala2019deep, hughes2020neural, levitt2019big, davies2021advancing, pawel2021knot}. More precisely, we use ideas introduced in the work of D{\l}otko, Gurnari, Hajij, Levitt, and the first author, applying manifold learning and topological data analysis, and compare the results of these two approaches \cite{levitt2019big, pawel2021knot}.

In this paper, Principal Component Analysis (PCA), a widely used dimension reduction technique, filtered PCA introduced in \cite{levitt2019big}, and TDA's Ball Mapper algorithm are applied to chromatic polynomial data for graphs with a fixed number of vertices. The dimensionality of chromatic data for graphs of all orders up to 10 appears to be one, with the first principal component influenced by all non-trivial chromatic coefficients. The corresponding Ball Mapper graph (BMGraph) has remarkable and intriguing resemblance with the PCA structure. Though the evidence is only empirical, the first principal component (PC1) corresponds to the elongated, linear structure of the Ball Mapper graph, and the second principal component (PC2) corresponds, informally, to its ``width". 

The structure of the chromatic data is further explored with respect to various graph invariants using Ball Mapper as a tool for visualizing functions between high dimensional spaces.  The linear structure of the  Ball Mapper, analogous to the first Principal component,  is determined by the number of edges in a graph, which is consistent with the prominence of edges $E$ in coefficient formulas and the PCA observations.  When colored by the corresponding values of PC2, the linear structure of the BMGraph corresponding to PC1  splits into concurrent/parallel strands  determined by increasing values of PC2. Similarly, for a fixed number of edges and vertices, the collection of clusters in the Ball Mapper graph corresponds to a slice of points in the 2-dimensional PCA projection. 
Clusters on both ends of each slice
in the PCA projection are 
graphs whose chromatic coefficients are extremal in the Rodriguez/Satyanarayana poset \cite{rodriguez1997chromatic}.

Our analysis facilitates the study of numerical graph invariants through dimensionality reduction and clustering. Recent work on graph compression summarized in Section \ref{compress} suggests a correlation between the chromatic polynomial and various measures of graph irregularity. While these measures are incomparable, our two-dimension projection of chromatic data, combined with Ball Mapper, reveals clusters of graphs which tend to maximize or minimize these invariants among all graphs with the same numbers of vertices and edges, see Section \ref{extrirr}.

For example, on one end of the spectrum defined by our second principal component PC2, we find clusters of  graphs with smallest chromatic coefficients and maximum irregularity scores, which tend to have more triangles, more cliques (complete subgraphs), and higher differences between vertex degrees; e.g., a  complete graph with one pendant edge. The other end of the spectrum features graphs with largest chromatic coefficients and minimum irregularity scores; which  tend to have a more regular structure, greater numbers of large cycles, and generally avoid cliques; e.g.,  Turan graphs as typical examples. 

This work leaves a number of natural questions unanswered and raises  more. Some of the questions to be addressed in the upcoming work are the following:

\noindent \textbf{Question 1} One of the main requirements in this approach is that the results are stable with respect to the filtration of our data, in this case by the number of vertices. It is quite remarkable that both dimensionality and Ball Mapper structure were stable, in addition to being consistent across methods. While filtered PCA \cite{levitt2019big} provides a reliable way of dimensionality reduction for data where finding a representative sample is a challenge, the next natural step is to apply these methods to a set of random graphs. 

\noindent \textbf{Question 2} Big data techniques, such as machine learning and topological data analysis have only recently been applied to data arising from theoretical mathematics where they could provide valuable insights and conjectures. This was achieved in knot theory \cite{davies2021advancing} where the relation between knot invariants was discovered by AI and later proven \cite{davies2021signature}. In this paper we provide insights into the missing theoretical characterization of chromatically maximal graphs, or the features corresponding to graph irregularity, which is quite elusive as the most prominent irregularity measures seem to be mutually incomparable \cite{abdo2019graph}. The scope and types of results in graph or knot theory that can be obtained using these approaches is an interesting open question.  

\noindent \textbf{Question 3} Methods used in this paper can be readily extended to other graph polynomials. The Tutte polynomial is particularly interesting as it specializes both to the chromatic and the Jones polynomial for knots. Determining the dimensionality of Tutte data is an exciting open problem since it should be at least 3, since that is the dimensionality of the Jones polynomial data \cite{levitt2019big}. Furthermore, Relational Ball Mapper or Mapper-on-Ball Mapper \cite{pawel2021knot}  constructions can be used to explore relations between the Tutte and chromatic, or Tutte and Jones polynomial. 

To summarize, this approach provides a new way of vectorizing graphs, 
and opens up another area of theoretical mathematics to big data techniques. In addition to immediate extensions of this work to other polynomials such as Tutte there are exciting but less clear goals to determining statistical nature of polynomial and other graph invariants and relations between them, as well as potential applications to other sciences.

\section*{Acknowledgements}
The authors extend their heartfelt gratitude to Pawel D{\l}otko and Davide Gurnari
for many helpful conversations and support with Ball Mapper. 
\section{Background}

In this section we provide necessary background in graph theory, focusing on the chromatic polynomial and irregularity measures. 

Let $G = (V, E)$ be a graph, with vertex set $V$ and edge multiset $E$. If $G$ has an edge between vertices $v_i, v_j \in V$, we write the corresponding element in $E$ as $\{v_i,v_j\}$.

\begin{definition}
The order of a graph  $G$ is the number of vertices $n = |V|$.
A loop in graph $G$ is an edge of the form $\{x, x\}$ for some $x \in V$. A multiple edge is an element that occurs multiple times in the multiset $E$. We say that $G$ is a simple graph if it has no loops or multiple edges.
\end{definition}

\begin{definition}
Let $v$ be a vertex of a simple graph $G$. The degree of vertex $v$, denoted $\deg(v)$, is the number of edges incident to $v$.
\end{definition}

\begin{definition}
A path in graph $G$ is a finite sequence of vertices $v_1, v_2, \ldots, v_n$ such that $\{v_i, v_{i+1}\}$ is an edge of $G$ for $i=1, 2, \ldots, n-1$ and all $v_i$ are distinct.
\end{definition}

\begin{definition}
A graph $G = (V, E)$ is connected if any two vertices in $G$ are joined by a path along edges of $G$. 
A bridge in a connected graph $G$ is an edge $e \in E$ such that the graph $(V, E - \{e\})$ is disconnected. A pendant edge is an edge incident to a vertex of degree one.
\end{definition}

\begin{definition}
The girth of graph $G$ is the length of the shortest cycle in $G$.
\end{definition}

\begin{definition}
Let $H = (V_H, E_H)$ be a subgraph of graph $G = (V, E)$. We say $H$ is an induced subgraph if for every $\{v_i,v_j\} \in E$ with $v_i,v_j \in V_H$, the edge $\{v_i,v_j\}$ is in $E_H$.
\end{definition}

\begin{definition}
The complete graph $K_n$ is the graph which consists of $n$ vertices and $\binom{n}{2}$ edges connecting each distinct pair of vertices.
A clique of size $m$ is a subset of vertices of $G$ whose induced subgraph is isomorphic to $K_m$. The clique number $\omega(G)$ is the maximum size over all cliques in $G$.
\end{definition}

\subsection{The chromatic polynomial}

\begin{definition}[\cite{whitney1932logical, Read1, DKT}] \label{ChromDef}
	Let $G = (V, E)$ be a graph. A mapping $f:V \to \{1, 2, \ldots, \la\}$ is called a $\la$-coloring of $G$ if for any pair of vertices $x,y \in V$ such that $\{x,y\} \in E$, $f(x) \neq f(y)$. Two $\la$-colorings $f, g$ are distinct if $f(x) \neq g(x)$ for some vertex $x \in V$. The chromatic polynomial of the graph $G$, denoted $P_G(\la)$, is equal to the number of distinct $\la$-colorings of $G$.
\end{definition}

The degree of the chromatic polynomial is equal to the order of $G$, and thus we write the polynomial with coefficients $c_{n-i}$ as follows:
\begin{align*}
    P_G(\la) = c_n\la^n +  c_{n-1}\la^{n-1} + c_{n-2}\la^{n-2} + \ldots + c_1 \la + c_0
\end{align*}
Since no graph has a $0$-coloring, the roots of $P_G(\la)$ include zero, which forces the free term to be zero $c_0=0$ for all graphs.

The chromatic polynomial of a disjoint union of graphs is the product of chromatic polynomials of its components. Additionally, adding multiple edges does not change the chromatic polynomial, while the existence of a loop trivializes it. Therefore, in this paper  we consider only simple connected graphs.

\begin{theorem}\cite{Meredith1} \label{Meredith}
If $G$ is a graph with $m$ edges, girth $g>2$, and $n_{g}$ cycles of length $g$, then the first $g$ coefficients of the chromatic polynomial are:
$$c_{n-i} = \begin{cases}
(-1)^i \displaystyle\binom{m}{i} & 0 \le i < g - 1\\
(-1)^{g-1} \left( \displaystyle\binom{m}{g-1} - n_{g}\right) & i = g - 1\\
\end{cases}$$

\end{theorem}

For a detailed summary of graph invariants which are determined by the chromatic polynomial see \cite{noy2003graphs}.

\begin{prop}[\cite{noy2003graphs}] \label{chrom_determine}
For a connected graph $G$, the chromatic polynomial $P_G(\la)$ determines the numbers of vertices, edges, triangles, and blocks in $G$, as well as the girth $g$ and the number of cycles of length $g$.
\end{prop}

It was long conjectured \cite{Read1} that the coefficients of $P_G(\la)$ are unimodal in the following sense: if $h_i = |c_{n-i}|$, there exists some $j$ with $2 \le j \le n-1$ such that
$$h_n \le h_{n-1} \le \ldots \le h_{j-1} \le h_j \ge h_{j+1} \ge \ldots \ge h_0.$$
This result has since been proven \cite{huh2012milnor} by showing that the stronger property of log-concavity holds for this sequence: $h_{i-1}h_{i+1} \le h_i^2$ for all $0<i<n$.

\begin{definition} A graph is chromatically unique if no non-isomorphic graph has the same chromatic polynomial.
\end{definition}
Cycle graphs, $\theta$-graphs \cite{CW1} and graphs which consist of wheel graphs with all but three or four adjacent spokes removed \cite{CW2} are families of chromatically unique graphs. 

There are also infinitely many graphs which are not chromatically unique; i.e., non-isomorphic graphs which share a chromatic polynomial. Examples include trees with $n$ vertices, whose chromatic polynomials all have the form $\lambda (\lambda-1)^{n-1}$ regardless of structure \cite[Theorem 13]{Read1}.
 
For any connected graph $G$, the first six coefficients of the chromatic polynomial have known formulas computed by counting subgraphs of $G$.

\begin{theorem}\label{Farrell}\cite{Farrell, Bielak1}
The first four coefficients of the chromatic polynomial $P_G(\la)$ are given by the following formulas: $c_n = 1$, $c_{n-1} = -m$, $c_{n-2} = \displaystyle\binom{m}{2}-t_1$, and $c_{n-3} = -\displaystyle\binom{m}{3}+(m-2)t_1+t_2-2t_3$, where $m$ is the number of edges in $G$, $t_1$ is the number of triangles, $t_2$ is the number of induced 4-cycles (pure squares with no diagonal present), and $t_3$ is the number of complete subgraphs on 4 vertices.

The 5th and 6th coefficients are given by the following formulas, where the $t_i$s count induced subgraphs of $G$ on 5 and 6 vertices specified in \cite{Bielak1}. For example, $t_4$ is the number of induced 5-cycles (pure pentagons) and $t_8$ is the number of induced graphs isomorphic to $K_5$.
\begin{eqnarray*}
c_{n-4} &=& \binom{m}{4} - \binom{m-2}{2}t_1 + \binom{t_1}{2} - (m-3)t_2 -(2m-9)t_3 - t_4 +t_5 + 2t_6+\\ && 3t_7-6t_8\\
c_{n-5} &=& -\binom{m}{5} + \binom{m-2}{3}t_1 - (m-4)\binom{t_1}{2} + \binom{m-3}{2}t_2 - (t_2-2t_3)t_1 +t_4\\ &-& (m^2-10m+30)t_3
-(m-3)t_5-2(m-5)t_6-3(m-6)t_7+6(m-8)t_8\\&+&t_9-t_{10}-2t_{11}-2t_{12}-t_{13}+t_{14}-t_{15}-3t_{16}-4t_{17}-4t_{18} +2t_{19}-4t_{20}\\&-&t_{21}+4t_{22}+3t_{23}+4t_{24}+5t_{25}+4t_{26}+6t_{27}+8t_{28}+16t_{29}+12t_{30}-24t_{31}
\end{eqnarray*}
\end{theorem}

Note that Whitney's original paper introducing the chromatic polynomial gives an interpretation of each $c_{n-i}$ in terms of spanning subgraphs of $G$ \cite{whitney1932logical}. Gian-Carlo Rota showed that each coefficient can be expressed as a M{\"o}bius function on a poset whose elements are subsets of $E$, and that the sum of the coefficients of $P_G(\la)$ is zero \cite{rota1964foundations}. Moreover, Meredith showed that the size of every coefficient is bounded by $|c_{n-i}| \leq \binom{m}{i}$ \cite{Meredith1}. These upper bounds are achieved for the first $g$ coefficients of the chromatic polynomial when all cycles of $G$ have length at least $g$.

In \cite{kahl2022extremal} the author defines a Tutte polynomial poset and characterizes the minimal graphs in this poset. With a similar motivation, we define a chromatic polynomial poset as follows.

\begin{definition}
Let $G(n, m)$ denote the set of simple connected graphs with n vertices and m edges. The chromatic polynomial poset is defined on $G(n, m)$ by $H \le G$ if and only if $P_G(\la) - P_H(\la) = r(x)$ for some polynomial $r(x)$ with non-negative coefficients. In other words, if we denote chromatic polynomial coefficients by $c_i$, then $H \le G$ if and only if  $|c_i(H)| \le |c_i(G)|$ for all $1 \le i \le n$.
\end{definition}

Rodriguez and Satyanarayana \cite{rodriguez1997chromatic} give a complete characterization of the graphs which are minimal in the chromatic polynomial poset.

\begin{definition}[\cite{DKT}]\label{minfams}
 Let $J(n,m)$ be the family of graphs $G \in G(n,m)$ such that all blocks of $G$ are complete graphs, including possibly triangles $K_3$ and double edges $K_2$.
 Let $L(n,m)$ be the family of graphs $G \in G(n,m)$ such that one block $B$ of $G$ has clique number $\omega(B) \ge |V(B)|-1$ and all other blocks are $K_2$s. 
\end{definition}

Note that $B$ is a complete graph or contains a complete subgraph containing all but one of the vertices of $B$, while all other blocks are bridges or pendant edges.

\begin{prop}[\cite{rodriguez1997chromatic}]
Any pair of graphs $G, H \in L(n,m) \cup J(n,m)$, have the same chromatic polynomials $P_G(\la) = P_H(\la).$
\end{prop}

\begin{theorem}[\cite{rodriguez1997chromatic}] \label{minchromcoeffs}
Any graph $H \in L(n,m) \cup J(n,m)$ is a minimal element of the poset $G(n,m).$

\end{theorem}

In Section \ref{extrirr}, we describe extremal graphs in our data set using the following definition based on the chromatic polynomial poset.

\begin{definition} \label{chromextremedef}
A chromatically minimal (resp., chromatically maximal) graph with $n$ vertices and $m$ edges is one which is a minimal (resp. maximal) element of the poset $G(n,m)$.
\end{definition}

\subsection{Irregularity measures for graphs} \label{irreg_bkground}
A graph $G$ is $k$-regular if $\deg(v) = k$ for each vertex $v \in V(G)$. Numerous invariants have been proposed to describe how far a given graph is from being regular, i.e., how irregular it is. A generally accepted criterion for an irregularity measure is that it should return a value of zero if and only if a graph is regular. The difference between the maximum and minimum vertex degrees in a graph is one example of an irregularity measure. Two other commonly cited measures are spectral irregularity and variance irregularity of a graph.

\begin{definition}[Spectral irregularity \cite{von1957spektren}] \label{spectralirrdef}
Let $\la_1(G)$ be the largest eigenvalue of the adjacency matrix of graph $G$ and let $\overline{d}(G)$ be the average degree of the vertices of $G$. The spectral irregularity measure of $G$ is $\epsilon(G) = \la_1(G)-\overline{d}(G)$, which is zero if $G$ is regular and positive otherwise.
\end{definition}

\begin{definition}[Variance irregularity \cite{bell1992note}] \label{varirrdef}
Suppose $G$ has order $n$ and the vertices are labeled, with $d_i = deg(v_i)$. The variance irregularity measure of $G$ is $$\sigma(G) = \dfrac{1}{n} \displaystyle\sum_{i=1}^n d_i^2 - \dfrac{1}{n^2} \left(\displaystyle \sum_{i=1}^n d_i \right)^2.$$
\end{definition}

For other proposed irregularity measures see \cite{albertson1997irregularity, nikiforov2006eigenvalues, abdo2014total} or a survey \cite{oliveira2013measures, dimitrov2014comparing} that also talks about  their comparative strengths.

The following classes of graphs are the most irregular under spectral and variance irregularity.

\begin{definition}[\cite{ahlswede1978graphs, oliveira2013measures}] \label{quasicomplete}
Let $G$ be a connected graph with $n$ vertices and $m$ edges. Let $d, t$ be the unique integer values with $2 \le d$ and $0 \le t < d$ such that $m = \binom{d}{2}+t$. Then $G$ is the quasi-complete graph $QC(n, m)$ if it is the unique graph such that:
\begin{enumerate}
    \item $G$ has a clique $C$ of size $n-1$;
    \item The single vertex $v'$ not included in $C$ is adjacent to exactly $t$ of the other vertices.
\end{enumerate}
\end{definition}

\begin{definition}[\cite{ahlswede1978graphs, oliveira2013measures}] \label{quasistar}
Let $G$ be a connected graph with $n$ vertices and $m$ edges. Let $d, t$ be the unique integer values with $2 \le d$ and $0 \le t < d$ such that $m = \binom{n}{2}-\binom{d}{2}-t$. Then $G$ is the quasi-star graph $QS(n, m)$ if it is the unique graph that:
\begin{enumerate}
    \item contains a  set $S$ of $n-d-1$ vertices, each of which is adjacent to all other vertices of $G$;
    \item contains one vertex adjacent to exactly $n-t$ vertices of $S$;
    \item the only edges of $G$ are the ones already specified.
\end{enumerate}
\end{definition}

\begin{prop}[Proposition 1, \cite{bell1992note}]
Among all graphs in $G(n, m)$, the maximum spectral irregularity is achieved by quasi-complete $QC(n,m)$, and the maximum variance irregularity is achieved by $QC(n,m)$ for $m > \frac{1}{2}\binom{n}{2} + \frac{n}{2}$ and by quasi-star fraphs $QS(n,m)$ such that $m < \frac{1}{2}\binom{n}{2} - \frac{n}{2}$.
\end{prop}

On the other hand, we have graphs which are as close to regular as possible, such as Turan graphs.

\begin{definition} \label{turan}
A Turan graph $T(n, r)$ is the graph formed by partitioning a set of $n$ vertices into $r$ groups, with sizes as equal as possible, and connecting two vertices by an edge if and only if the vertices belong to separate groups. If we let $t, p \in \mathbb{Z}$ such that $n = tr+p$, each vertex of $T(n, r)$ has degree either $n-t$ or $n-t-1$.
\end{definition}

 Turan graphs are a family of extremal graphs whose degree sequences are constant or nearly constant   and they are characterized by the following theorem. 
 
 \begin{theorem}[Turan's Theorem, \cite{turan1941external}] \label{turantheorem}
Turan graph on n vertices with r groups $T(n, r)$ has the maximal number of edges among all order $n$ graphs that contain no $(r+1)$-clique.
\end{theorem}

\subsection{Threshold graphs and compression} \label{compress}

Compression of a graph is an operation first introduced in \cite{kelmans1981graphs} and known by a variety of names, including the Kelmans transformation \cite{csikvari2011applications} and swing surgery \cite{satyanarayana1992reliability}. See the introduction of \cite{kahl2022extremal} for details.

\begin{definition}
The neighborhood of vertex $v$ in graph $G$, denoted $N_G(v)$, is the set of all vertices which are connected to $v$ by an edge. We say that vertex $v$ dominates vertex $u$ if $N_G(u) \subset N_G[v]$, where $N_G[v] = N_G(v) \cup \{v\}$.
\end{definition}

\begin{definition}\label{tresh}
A graph $G$ is threshold if for all vertices $u, v \in V(G)$, either $u$ dominates $v$ or $v$ dominates $u$.
\end{definition}
For other characterizations of threshold graphs, see \cite{csikvari2011applications, keough2016graphs}.

\begin{definition}
Let $G = (V, E)$ be a graph and let $u, v \in V$. The compression of $G$ from vertex $u$ to vertex $v$ is an operation defined as follows: for each vertex $x \in N_G(u) - N_G(v) - \{v\}$ we delete all edges of the form $\{x, u\}$ from $G$ and replace them with corresponding edges of the form $\{x, v\}$. The new graph produced by the compression of $G$ from $u$ to vertex $v$ is denoted $G_{u \rightarrow v}$.
\end{definition}

Compression preserves the number of vertices and edges of a graph, and the result is independent of the order of in which vertices $u$ and $v$ are used, up to isomorphism (i.e., $G_{u \rightarrow v}$ and $G_{v \rightarrow u}$ are isomorphic graphs). Compression of a graph monotonically increases a number of graph invariants, such as the spectral radius and the number of independent sets of order $k$, while decreasing other invariants such as the number of spanning trees and the vertex connectivity \cite{csikvari2011applications, kahl2021graph, kahl2022extremal}. 

Any connected graph can be transformed into a connected threshold graph by repeated applications of graph compression \cite{bogdanowicz1985spanning, satyanarayana1992reliability}. Thus, threshold graphs are often extremal with respect to these types of invariants.

\section{Methods and data}

\subsection{Principal component analysis} \label{PCA}

Principal component analysis (PCA) is a common technique for exploring multivariate data. The algorithm finds a set of directions within the feature space which explain most of the variance in the data. The number of these directions is often far smaller than the number of original features, and so PCA can be used as a method of dimensionality reduction when visualizing and interpreting a high-dimensional data set.

Let $X \in \R^{n \times p}$ be a matrix representing a data set of $n$ individuals, each with $p$ features. PCA takes the sample covariance matrix $S$ for the data set and finds an orthonormal eigensystem $\{(\la_j, a_j)\}_{j=1}^p$ for $S$ with eigenvalues $\la_1 \ge \la_2 \ge \ldots \ge \la_{p}$. The eigenvector $a_1$ is the unit vector in $\R^p$ such that the entries of $Xa_1$ have maximal variance (up to an overall sign change for $a_1$). The second eigenvector $a_2$ maximizes variance subject to the constraint that $a_2$ is orthogonal to $a_1$, and all successive vectors $a_j$ maximize variance among directions orthogonal to $\{a_1, \ldots, a_{j-1}\}$.

\begin{definition}\cite{jolliffe2016principal}
Given a data set $X \subset \R^{n \times p}$ and PCA eigensystem $\{( \la_j, a_j)\}_{j=1}^p$, the $j$th principal component of $X$ is the product $Xa_j$. For the $j$th principal component of $X$, the $p$ elements of the vector $a_j$ are called the PC loadings for the $j$th principal component, and the $n$ elements of $Xa_j$ are called the PC scores.
\end{definition}

Note that if $a_j$ is an eigenvector of $S$ with eigenvalue $\la_j$, then so is its opposite $-a_j$. Thus the signs of PC loadings and PC scores are arbitrary and may be reversed depending on the implementation of PCA in a given instance.

\begin{definition}\cite{jolliffe2016principal}\label{expvardef}
The explained variance associated with the $j$th principal component of the data set $X$ is the eigenvalue $\la_j$, and the normalized explained variance $\overline{\la_j}$ is defined by $\overline{\la_j} = \dfrac{1}{\sum_{i=1}^p \la_i} \la_j$. 

The proportion of variance explained by the first $j$ principal components of $X$ will be denoted $S_j = \sum_i^j \overline{\la_i}$. 
\end{definition}

The eigenvector $a_1$ is a unit vector that determines a line in the feature space $\R^p$. This line indicates the axis in $\R^p$ along which the data set represented by $X$ has the greatest variation. Similarly, $a_2$ gives the direction orthogonal to $a_1$ for which the data set has the second-greatest variation, and so on for $ 2< j \le p$.

Suppose that $x \in \R^p$ is a column of $X$ representing one individual in the data set. The PC score for the $j$th principal component represents the length of the projection of $x$ onto the line corresponding to eigenvector $a_j$. The implementation of PCA in the Python {\tt scikit-learn} package \cite{scikit-learn} centers the data by moving the mean to the origin. Thus a PC score of zero in the $j$th component would indicate that $x$ lies at the center of the distribution with respect to the $j$th direction of variation. A positive or negative score indicates that $x$ lies further from the center on either side along this direction.

\subsection{Ball Mapper}

\begin{figure}{}
    \begin{subfigure}[]{0.23\textwidth}
    \includegraphics[height=0.8\textwidth]{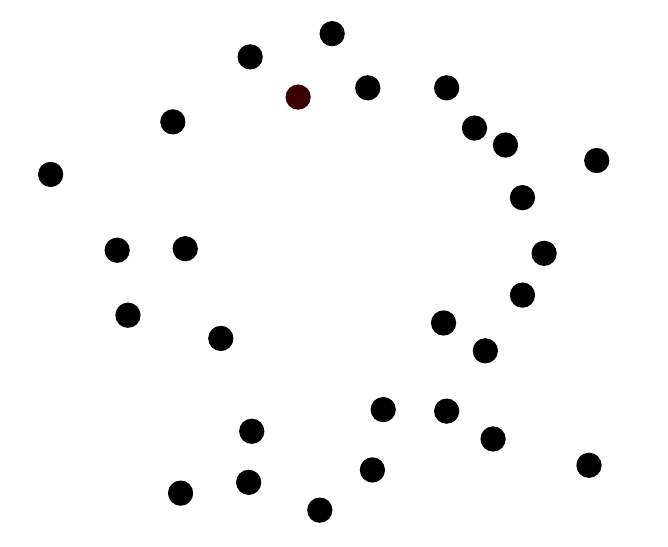}
         \caption{}
     \end{subfigure}\hspace{0.2cm}
     \begin{subfigure}[]{0.23\textwidth}
     \includegraphics[height=0.9\textwidth]{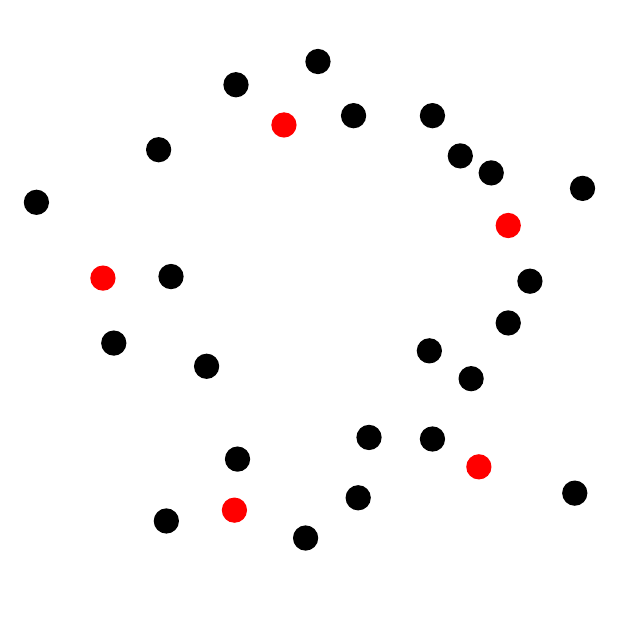}
         \caption{}
     \end{subfigure}\hspace{0.1cm}
       \begin{subfigure}[]{0.23\textwidth}   \includegraphics[height=0.9\textwidth]{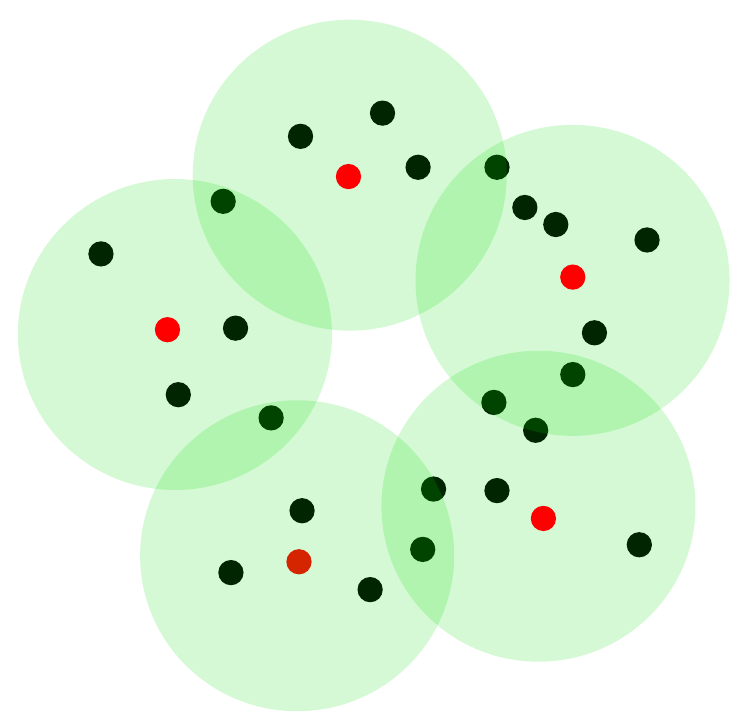}
         \caption{}
     \end{subfigure}
     \begin{subfigure}[]{0.23\textwidth}
     \includegraphics[height=0.85\textwidth]{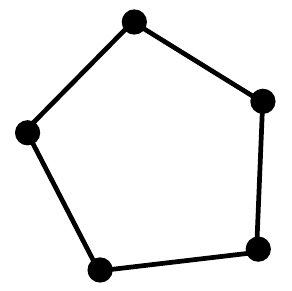}
         \caption{}
     \end{subfigure}

    \caption{Ball Mapper construction: Illustration. The input point cloud is shown in (A), $\epsilon$--net (B), Ball Mapper cover (C), corresponding Ball Mapper graph (D). }
    \label{fig:bm_example}
\end{figure}

Ball Mapper \cite{dlotko2019ball, ball_mapper} is a recent Mapper-type algorithm for visualizing and analyzing the local and global structure of a data set. Like the widely-used Mapper algorithm \cite{singh2007topological}, Ball Mapper takes a data set in $\R^n$ as an input and returns a graph which describes the proximity of points within the data. While Mapper is highly dependent on a choice of filter function for the data set and a choice of clustering method, Ball Mapper's approach to obtaining the cover is simpler and requires only one input parameter, a radius $\epsilon$.

Given a set of points $X \subset \R^d$ and $\epsilon > 0$, see Figure \ref{fig:bm_example}(A),  Ball Mapper finds a collection $C \subset X$, called an $\epsilon$--net Figure \ref{fig:bm_example}(B), such that no point in $X$ is located further than distance $\epsilon$ from an element of $C$. This collection determines an overlapping cover $B(X, \epsilon) = \{B_1, B_2, \ldots, B_i, \ldots\}$: a set of $d$-dimensional balls with radius $\epsilon$ that contain all of $X$, Figure \ref{fig:bm_example}(C). The algorithm then constructs the nerve of $B(X, \epsilon)$, called Ball Mapper graph (BMGraph) Figure \ref{fig:bm_example}(D). In a Ball Mapper graph, vertices  $v_i$ correspond to the balls $B_i$ in the cover, and two vertices $v_i, v_j$ are connected by an edge if and only if $B_i$ and $B_j$ contain shared points in $X$.

The Ball Mapper parameter $\epsilon$ determines the scale of the features highlighted in the output. By varying $\epsilon$, we create a linear sequence of images which highlight structures at successive scales, inspired by the filtration in persistent homology (but in this case we do not have the stability result). Additionally, we can choose to filter our data according to a range of values of a given criterion and obtain an output that can be displayed as a sequence of graphs,  see Figure \ref{fig:ch_vertices_pers}.

Results in this paper are obtained using the R implementation of Ball Mapper \cite{dlotko2019ballmapper}. This version of Ball Mapper supports additional visualization features; e.g. the disks corresponding to vertices of the Ball Mapper graph are scaled according to the number of points they contain. Thus larger disks in the visualization correspond to larger clusters of points within radius $\epsilon$ of each other. Another useful feature is coloring the Ball Mapper graph using values of some function  $f:X \to \R$ defined on the input data. The color of each Ball Mapper disc is the average value of the functions $f$ on the points of $X$ in that cluster. For convenience, we include the colored vertical bar on the right of each image which describes a uniform scale illustrating the distribution of colors from the minimum to maximum values of $f$ within the data set.

\subsection{Graph data} \label{DataBkgrd}

The focus of this paper is on the chromatic polynomials and our data contains chromatic polynomials for all simple connected graphs up to ten vertices for a total of 11,716,571 graphs. Our database contains many additional properties including numbers of vertices and edges, spectral irregularity, variance irregularity, numbers of subgraphs including cycles and $K_4$s, maximal and minimal vertex degrees, and identifications of all graphs which belong to classes including trees, Turan graphs, or bipartite graphs.

Chromatic polynomials are computed  using the Sage \texttt{GraphGenerators} class implemented in Python 3 \cite{sage, conda_forge}. The point cloud consists of the set of all chromatic polynomial coefficients obtained in the following way. Given graph $G$ with chromatic polynomial $P_G(\la) = c_n \la^n + c_{n-1} \la^{n-1} + \ldots + c_1 \la$, we represent $P_G(\la)$ by its coefficient vector $Q(G) = (|c_{n}|, |c_{n-1}|, \ldots, |c_1|, 0, \ldots, 0)$. For simplicity and computational efficiency, we take absolute values of each coefficient, since the sign is given by $(-1)^{k}$ for each coefficient $c_{n-k}$ and thus remains constant within each feature of our data set. Our data set consists of coefficient vectors with at most 10 non-zero entries, so we standardize length by adding zero entries for graphs with less than $10$ vertices. Table \ref{examplechromvectors} contains  sample vectors for the triangle $C_3$, the complete graph $K_5$, and the cycle $C_{10}$ with 10 edges. 

\begin{table}
\renewcommand{\arraystretch}{1.4}
\begin{center}
\resizebox{\columnwidth}{!}{   \begin{tabular}{ | p{0.75cm}|p{7cm}|p{6.2cm}| }
    \hline
     $\mathbf{G}$ & Chromatic polynomial $\mathbf{P_G(\la)}$ & Vector $\mathbf{Q}(G)$\\ \hline
    $C_3$ & $\la^3 - 3\la^2 + 2\la$ & $(1, 3, 2, 0, 0, 0, 0, 0, 0, 0)$ \\ \hline
    $K_5$ & $\la^5 - 10\la^4 + 35\la^3 - 50\la^2 + 24\la$ & $(1, 10, 35, 50, 24, 0, 0, 0, 0, 0)$ \\ \hline
    $C_{10}$ & $\la^{10} - 10\la^9 + 45\la^8 - 120\la^7 + 210\la^6 - 252\la^5 \newline + 210\la^4 - 120\la^3 + 45\la^2 - 9\la$ & $(1, 10, 45, 120, 210, 252, 210, 120, 45, 9)$ \\ \hline
  \end{tabular}}
  \caption{Chromatic point cloud examples: chromatic coefficient vectors (absolute values) for the triangle, complete graph on 5 vertices, the cycle graph with all ten non-zero entries. }
  \label{examplechromvectors}
  \end{center}
\end{table}

\begin{figure}[ht!]
     \begin{subfigure}[]{0.45\textwidth}
         \includegraphics[width=0.8\textwidth]{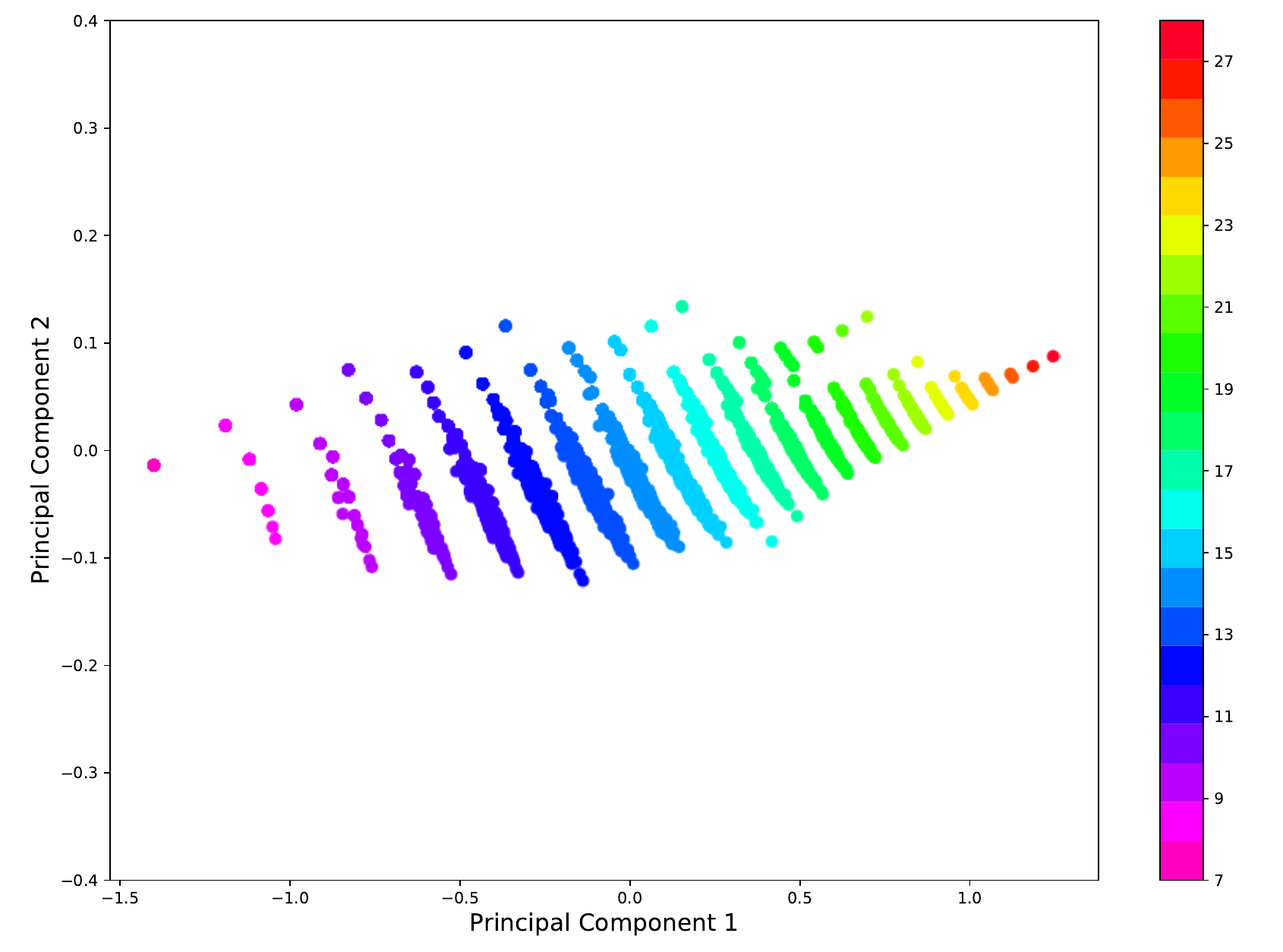}
         \caption{}
     \end{subfigure}
     \begin{subfigure}[]{0.45\textwidth}
         \includegraphics[width=0.8\textwidth]{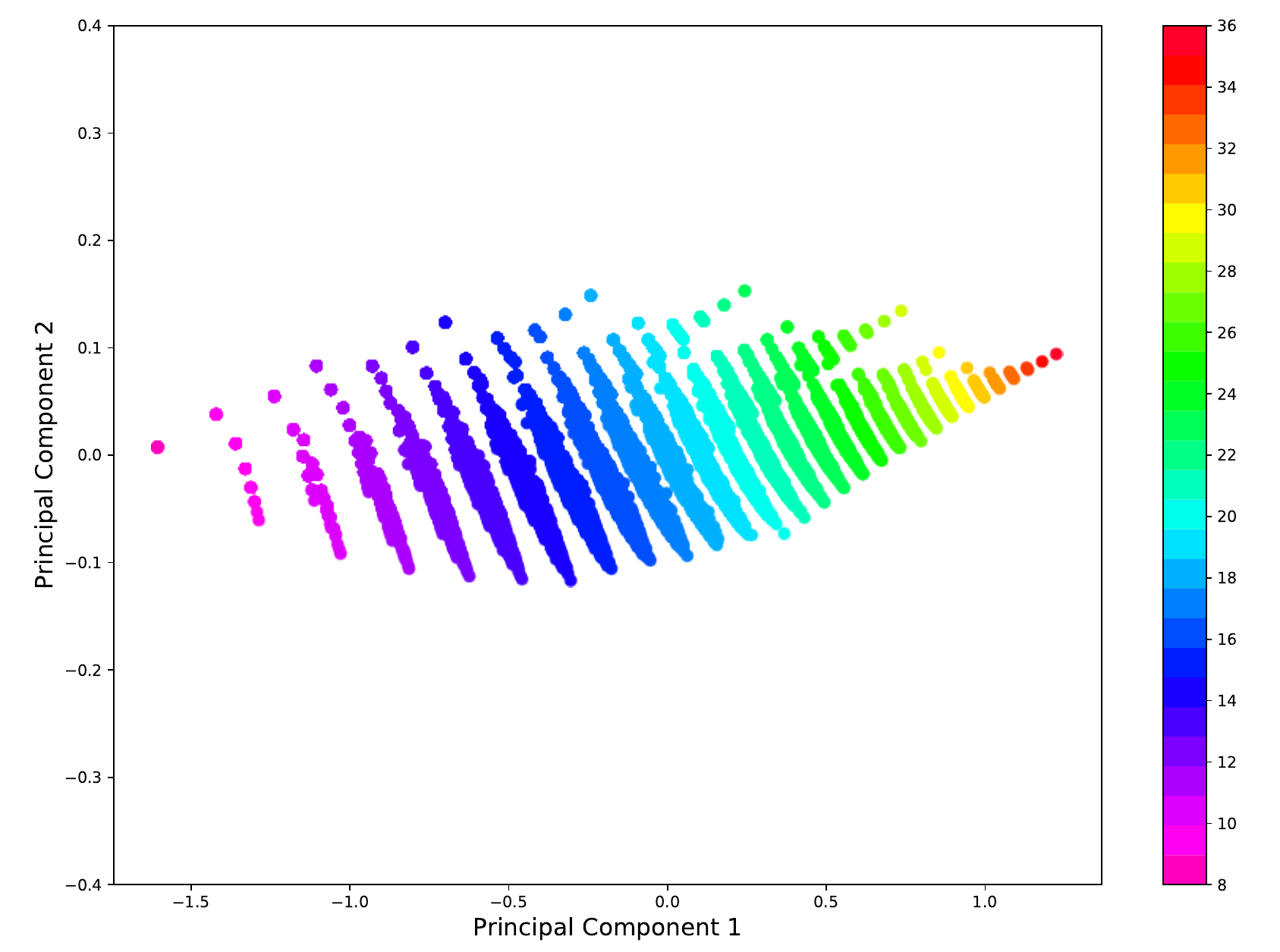}
         \caption{}
     \end{subfigure}
     \caption{A PCA projection of the chromatic polynomial data to 2 dimensions determined by the first two principal components and colored by number of edges for graphs with eight (A) and nine vertices (B).}
\label{fig:chPCAedges}
\end{figure}

In the sections that follow, we visualize data sets corresponding to graphs with a fixed number of vertices. Within each data set, the varying sizes of chromatic coefficients make it difficult to see important features during visualization. We apply a log transformation $f(x) = \ln(1+x)$ to the data set derived from graphs with $n$ vertices and then separately normalize each feature to the range $[0,1]$ with min-max scaling over this particular data set.

\begin{figure}[ht!]
         \includegraphics[width=0.4\textwidth]{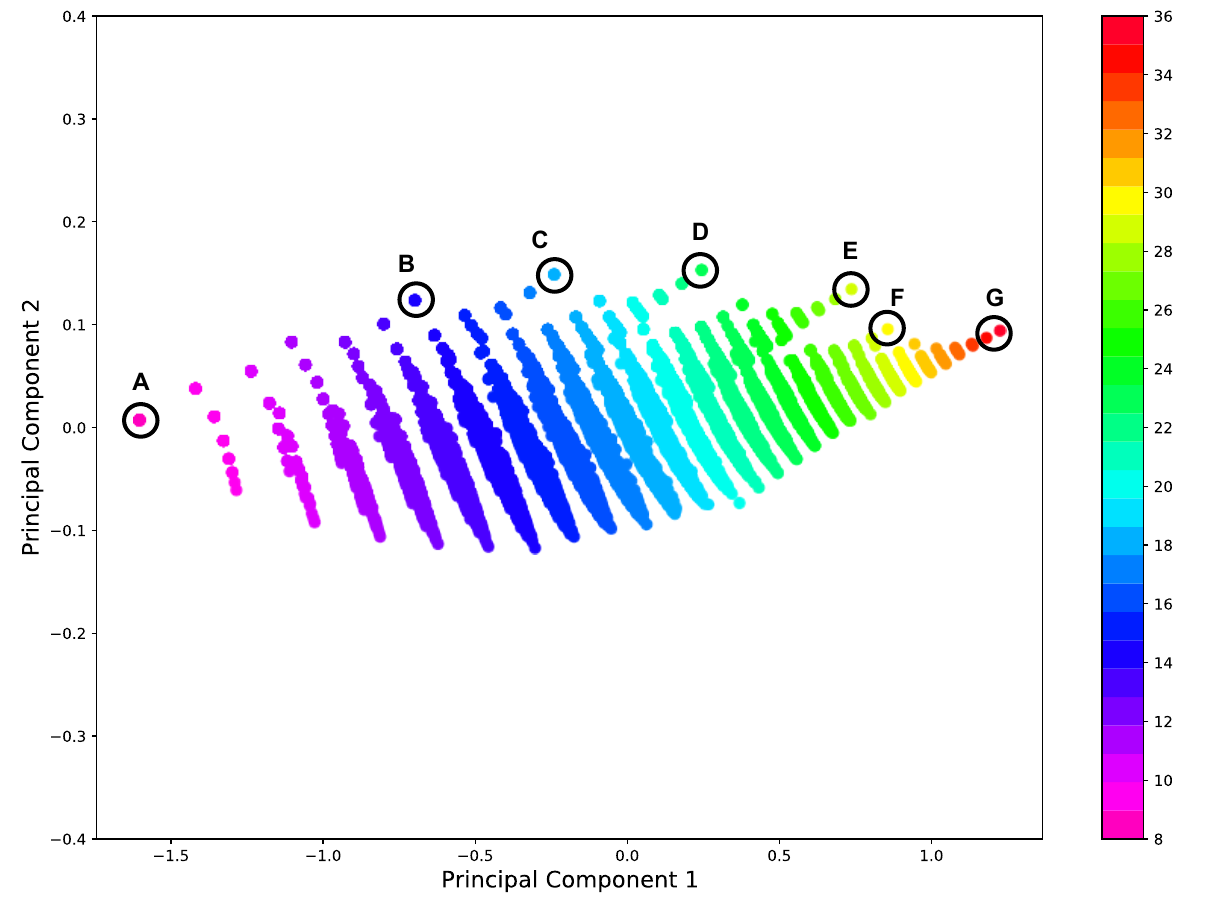}
     \caption{The PCA projection from Figure \ref{fig:chPCAedges}(B) with clusters containing some well known graphs circled and labeled. In addition to multiple non-isomorphic graphs the following clusters contain A: all 9-vertex trees, B: complete graph $K_5$ with 4 pendant edges, C: complete graph $K_6$ with 3 pendant edges, D: complete graph $K_7$ with 2 pendant edges, E: complete graph $K_8$ with 1 pendant edge, F: complete graph $K_8$ with triangle attached along 1 edge, G: complete graph  $K_9$, and $K_9$ with 1 edge removed. 
     }
\label{fig:chPCAedgesIdentify}
\end{figure}

\section{Chromatic polynomial data: structure and dimensionality}

The main results of this paper rely on introducing big data analysis to graph theory in a novel way inspired by applications to knot theory \cite{pawel2021knot, levitt2019big}. In these approaches a graph is represented by a unique vector, consisting of coefficients of the chromatic polynomial.  In this paper we focus on filtered PCA and Ball Mapper, but this graph representation lends itself to any of the big data techniques, including machine learning and other artificial intelligence tools.

Input data consists of chromatic polynomials for graphs up to ten crossings as described in Section \ref{DataBkgrd}. This point cloud naturally embeds in $\mathbb{R}^9$ and we are using Euclidean distance as our metric for this initial analysis.

Note that Theorem \ref{Farrell} provides an interpretation of the first few chromatic coefficients in terms of number of vertices and edges, as well as the number of special subgraphs such as cycles, complete graphs, and many others. Moreover, Bollob{\'a}s, Pebody and Riordan conjecture that almost every pair of independently chosen random graphs are distinguished by their Tutte and chromatic polynomials.

\begin{conj}[\cite{bollobas2000contraction}]
Let $G(n, 1/2)$ be the space of random graphs on a fixed set of n vertices in which each pair of vertices is joined independently with probability $1/2$. Let $T(G), p(G)$ denote the Tutte polynomial and chromatic polynomial of $G$ respectively. Then almost every graph $G \in G(n, 1/2)$ is such that $T(G')=T(G)$ implies $G'$ is isomorphic to $G$.
In addition, almost every graph $G \in G(n, 1/2)$ is such that $p(G')=p(G)$ implies $G'$ is isomorphic to $G$.
\end{conj}

While the conjecture remains open, it has been shown that the probability that two independently chosen random graphs have the same Tutte polynomial is of order $O(1/\log n)$ \cite{loebl2004triangles}.

\subsection{Dimension of the chromatic data}\label{PCAdim}

Inspired by the big data approach to the Jones polynomial, a polynomial invariant for knots \cite{levitt2019big}, we apply Principal Component Analysis (PCA) to the chromatic data for graphs with increasing number of vertices, up to ten maximum. The \emph{dimension} of the Jones polynomial is defined to be the smallest value $\mathbf{d}$ for which the normalized explained variance of the first $\mathbf{d}$ principal components sums to more than 95\% across all considered orders \cite{levitt2019big}.

According to this definition and PCA computations, chromatic polynomial data is essentially one-dimensional, and this dimensionality is stable with respect to  increasing the order of a graphs, see Table \ref{PCAStable1}. This table presents a snapshot of PCA computations  for chromatic polynomial data for graphs of orders $n = 8, 9, 10$ and the first three principal components for each data set. In particular the last two rows in Table \ref{PCAStable1} contain  the normalized (NEV) and cumulative explained variance (CEV) for $n=8,9,10$ indicating that almost all variation is explained by the first  principal component. More precisely, the cumulative explained variance of PC1, $S_1$, is greater than $0.99$ for all orders up to 10.
In Section \ref{chromBM} we revisit this question since the structure of the Ball Mapper graph is linear, see comparison with the PCA in Figure \ref{fig:ch_9_zoom1_pca}.

\begin{table}[h]
\renewcommand{\arraystretch}{1}
\begin{center}
\resizebox{\columnwidth}{!}{  \begin{tabular}{|l||c|c|c||c|c|c||c|c|c|| }
    \hline
  &\multicolumn{3}{c||}{$n=8$} 
    & \multicolumn{3}{c||}{$n=9$}
    & \multicolumn{3}{c||}{$n=10$}\\ \hline
     & PC1 & PC2& PC3 &
  PC1 & PC2& PC3 &
PC1 & PC2& PC3 \\ \hline
 NEV     & 0.99188
           & 0.00792 & 0.00018 & 0.99178
   & 0.00806 & 0.00015 & 0.99189  & 0.00796 & 0.00015 \\ \hline
 CEV     & 0.99188
           & 0.99980 & 0.99998 & 0.99178
   & 0.99984 & 0.99999 & 0.99189  & 0.99985 & 0.99999 \\ \hline

  \end{tabular}}
  \caption{The normalized explained variance   $\overline{\la_j}(n)$ (NEV) and cumulative normalized explained variance $S_j(n)$ (CEV) for the first three  principal components of the chromatic polynomial data of graphs with orders  8, 9, and 10 respectively. }
  \label{PCAStable1}
  \end{center}
\end{table}

\subsection{Graph-theoretic interpretation of the principal components }

To get an insight into the graph-theoretic relevance of the principal components and the dimension reduction of our data from 10 to one (or two), we provide visualisations of the first two principal directions and further insights obtained using Ball Mapper; see Figure \ref{fig:ch_9_zoom1_pca}.

Focusing on the two most significant principal components on chromatic polynomials of graphs with fixed order, we visualize the structure of each data set by plotting each graph as a coordinate pair $(x_1,x_2)$ pair where $x_i$ represents the $i$th PC score of that graph, see Figure \ref{fig:chPCAedges}. This figure illustrates the persistent structure found for graphs with eight and nine vertices and the stability of the dimension reduction with respect to the number of vertices of a graph. The first principal component corresponds to the horizontal axis in Figure \ref{fig:chPCAedges}, which appears to be roughly correlated with an increase in the number of edges. Graphs with a fixed number of edges are found in clusters separated by thin regions of empty space in this projection. Figure \ref{fig:chPCAedgesIdentify} identifies the location of certain well-known graphs, such as trees and complete graphs, at the boundary of the projection. This boundary also contains maximum coefficient graphs; see Figure \ref{fig:chPCA9features1}.

The data set of chromatic polynomials for graphs of 9 vertices has 8 dimensions, corresponding to the eight coefficients which are not trivially equal to 0 or 1. The loadings for PC1 and PC2 in Table \ref{PCAloadingtable} are the entries of the eigenvectors that give the directions of greatest variation in this eight-dimensional data set. We observe that all loadings for PC1 are positive, so that movement along this axis corresponds to an overall increase or decrease in coefficient size. This agrees with our previous observation that PC1 is correlated with the number of edges in a graph. For PC2, the loadings for the last four coefficients have the opposite sign from the first four, indicating that an increase in PC2 score correlates with an decrease in relative size of the last few coefficients.

\begin{table}[h]
\renewcommand{\arraystretch}{1.4}
\begin{center}
\resizebox{\columnwidth}{!}{   \begin{tabular}{|l|c|c|c|c|c|c|c|c| }
    \hline
     & $c_8$ & $c_7$& $c_6$ &
  $c_5$ & $c_4$ & $c_3$ &
$c_2$ & $c_1$\\ \hline
 PC1    & 0.34694 & 0.34683 & 0.34761 & 0.34876 & 0.35005
          & 0.35172 & 0.35579 & 0.379545
 \\ \hline
 PC2   & 0.54568 & 0.39243 & 0.24013 & 0.08614
          & $-$0.06850 & $-$0.22145 & $-$0.37150
  & $-$0.53983 \\ \hline

  \end{tabular}}
  \caption{PC loadings for two most significant principal components, PC1 and PC2, for the chromatic polynomial data for all graphs of order 9, with chromatic polynomial determined by eight coefficients $P_G(\la) = \la^9 +  c_8\la^8 + \ldots + c_1 \la$}
  \label{PCAloadingtable}
  \end{center}
\end{table}

PC loadings provide a way to think about principal components in terms of the coefficients of the chromatic polynomial, and in turn, based on the formulas in Theorem \ref{Farrell}, features of a graph. We hypothesize that PC1 measures the average size of the coefficient vector as a whole, and that PC2 detects the relative difference between the first and last coefficients.

\begin{figure}[ht!]
\centering
    \begin{subfigure}[]{0.32\textwidth}
         \includegraphics[width=0.9\textwidth]{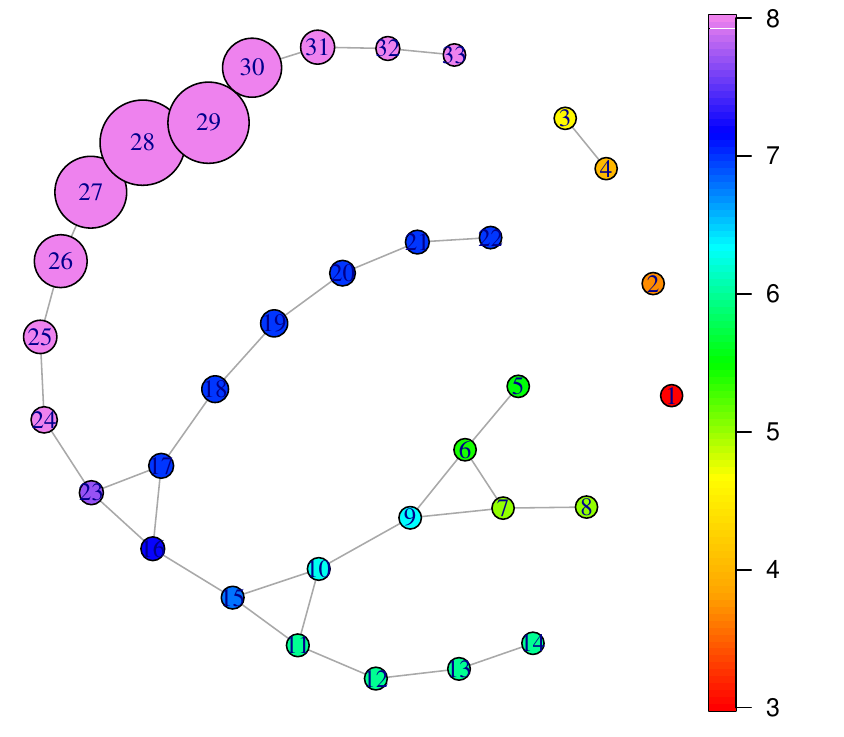}
         \caption{}
     \end{subfigure}
    \begin{subfigure}[]{0.32\textwidth}
         \includegraphics[width=0.9\textwidth]{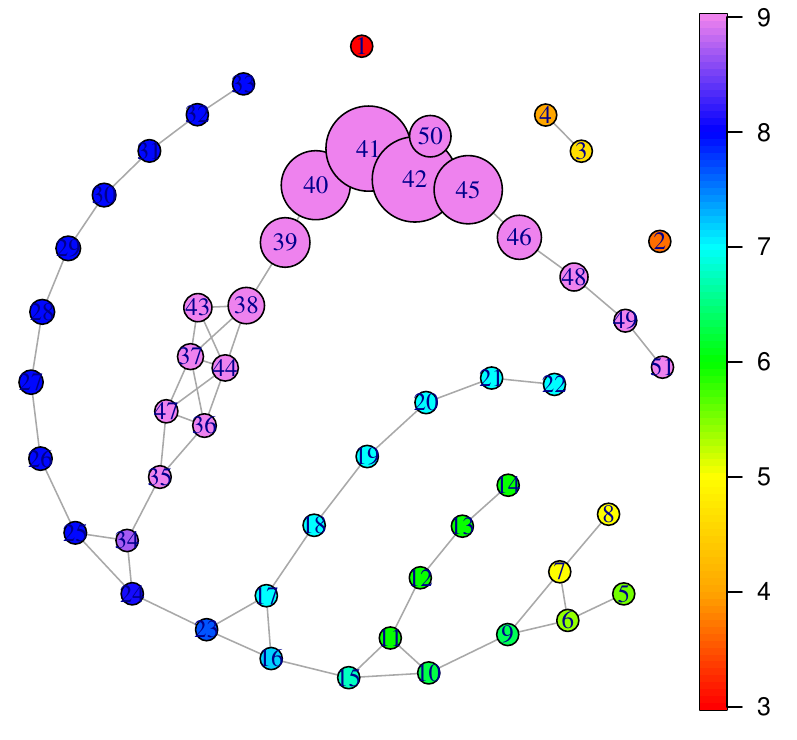}
         \caption{}
     \end{subfigure}
    \begin{subfigure}[]{0.32\textwidth}
         \includegraphics[width=0.9\textwidth]{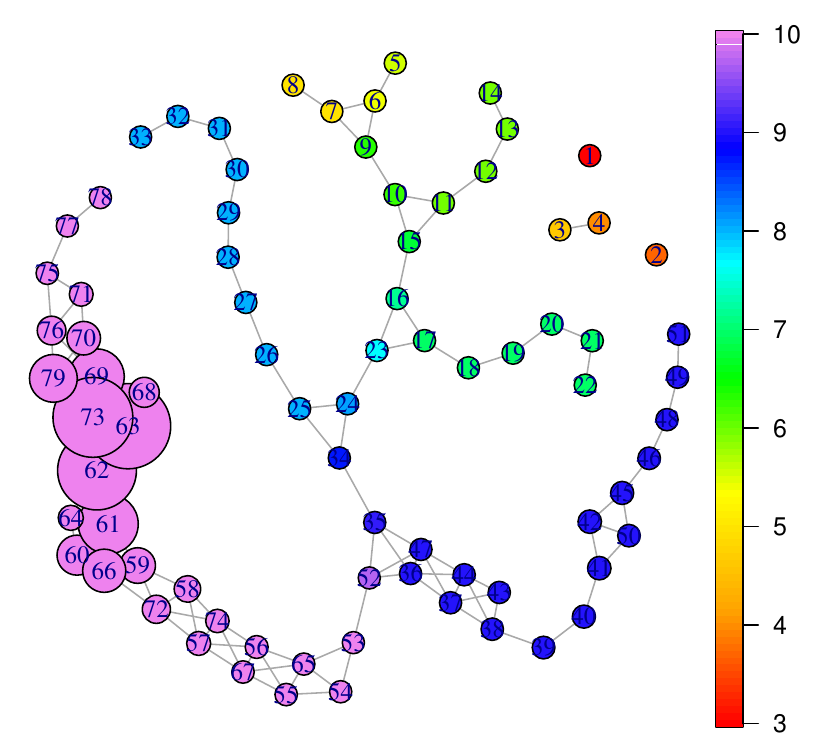}
         \caption{}
     \end{subfigure}
\caption{Stability of BMGraphs for chromatic data with respect to the number of vertices with graphs of order at most 8 (A), nine (B) and ten (C) with the Ball Mapper parameter  $\epsilon = 0.1$,  colored by the number of vertices.} 
\label{fig:ch_vertices_pers}
\end{figure}

\begin{figure}[ht!]
\centering
\includegraphics[width=0.7\textwidth]{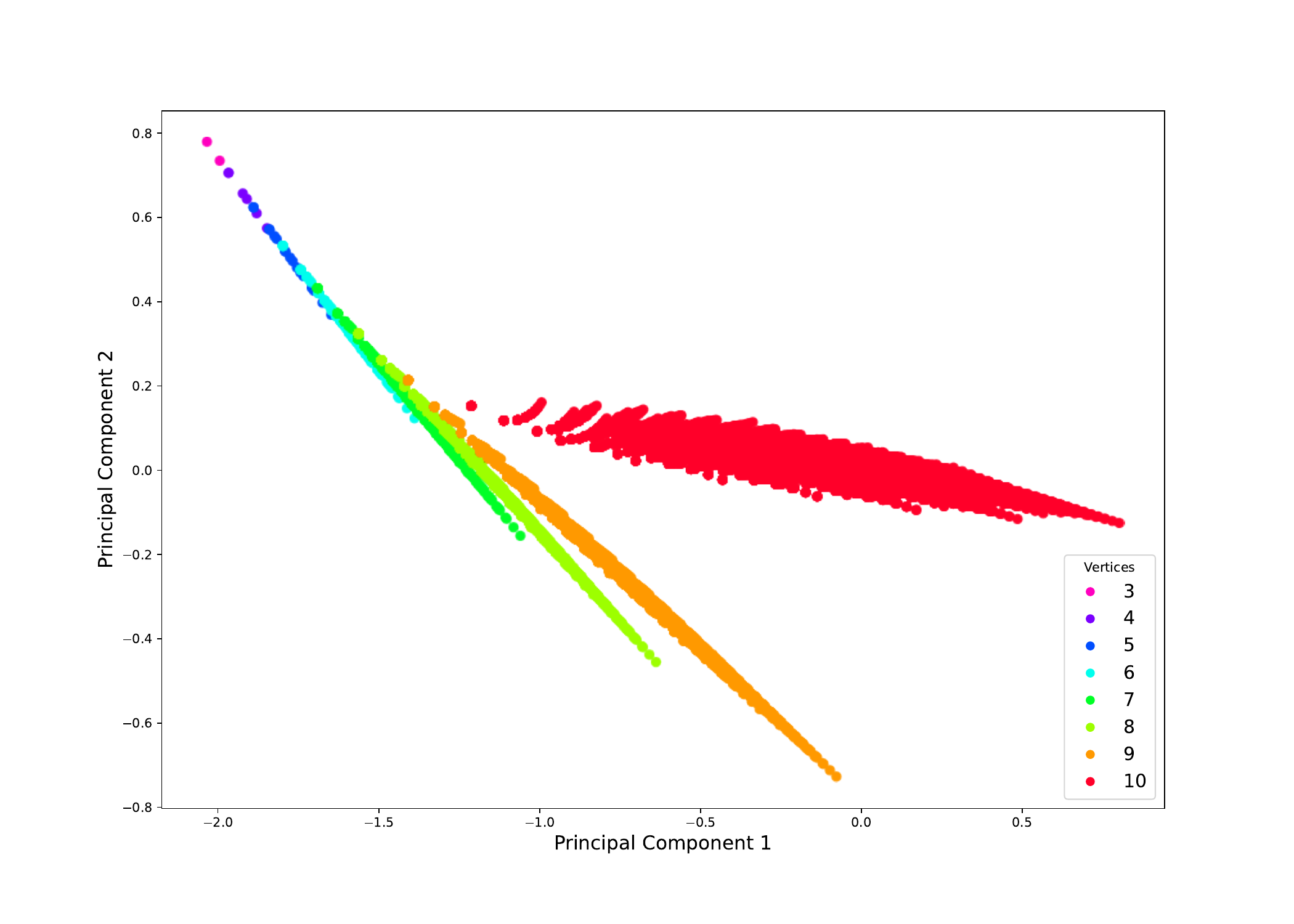}
   
\caption{A two-dimensional PCA projection of the chromatic polynomial data for all graphs with 3-10 vertices. The projection is based on the first two principal components for normalized data, colored by number of vertices.} 
\label{fig:ch_vertices_persPCA}
\end{figure}

\subsection{Chromatic data: Ball Mapper structure} \label{chromBM}

In this section we apply Ball Mapper to the  chromatic point cloud consisting of  coefficient vectors of all graphs up to $n = 10$ vertices. 

First, we provide the evidence for the stability of the Ball Mapper approach to the chromatic polynomial data with respect to the order of a graph. 
In Figures \ref{fig:ch_vertices_pers} and \ref{fig:ch_vertices_persPCA} we provide evidence for a persistent structure across data sets of graphs with order equal to eight, nine, and ten respectively. Graphs of each given order are clustered in long branch-like features, with neighboring branches of adjacent orders across all three BMGraphs. These branch-like features are also reflected in the two-dimensional PCA plot for all graphs up to 10 vertices combined, see Figure \ref{fig:ch_vertices_persPCA}. This structure invites and justifies exploring the data corresponding to graphs of a given order. 

Figure \ref{fig:ch_8branch} and Figure \ref{fig:ch_9branch} provide a detailed structure of branches consisting only of graphs with exactly 8 and 9 vertices respectively. Both Ball Mapper graphs have linear structure consistent with the PCA analysis in Section \ref{PCAdim}. To confirm the informal correspondence between the 1-dimensionality of the data obtained by the PCA approach and the ``linear" structure of the BMGraph,  we color each cluster in the BMGraph by the average PC scores of the graphs it contains, Figure \ref{fig:ch_9_zoom1_pca}. The remarkable agreement between the two methods suggests that the first principal direction corresponds to the linear feature of the Ball Mapper structure.

\begin{figure}[ht!]
\centering
  \begin{subfigure}[]{0.45\textwidth}\centering
         \includegraphics[width=0.7\textwidth]{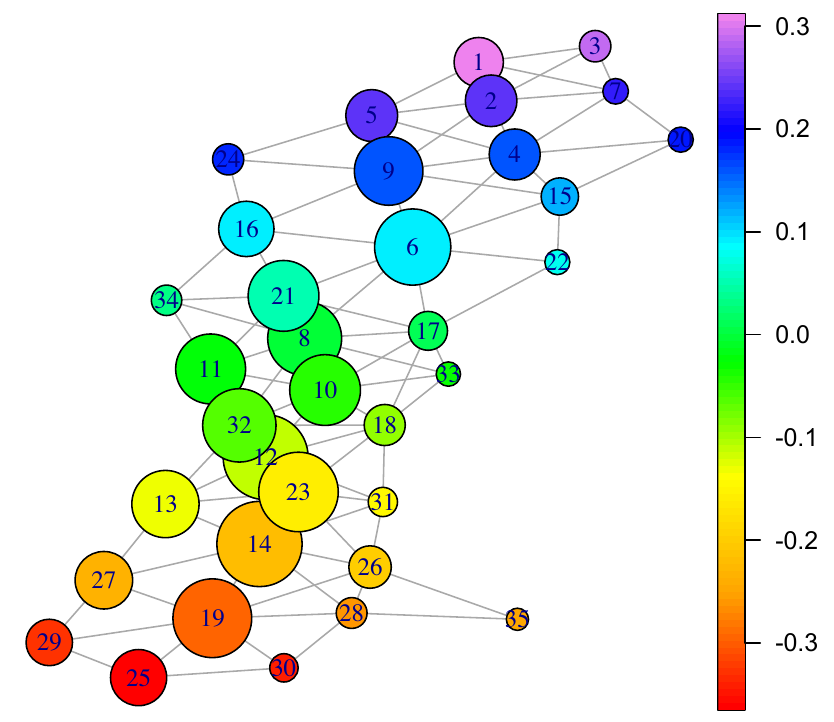}
         \caption{}
     \end{subfigure}
  \begin{subfigure}[]{0.45\textwidth}\centering
         \includegraphics[width=0.7\textwidth]{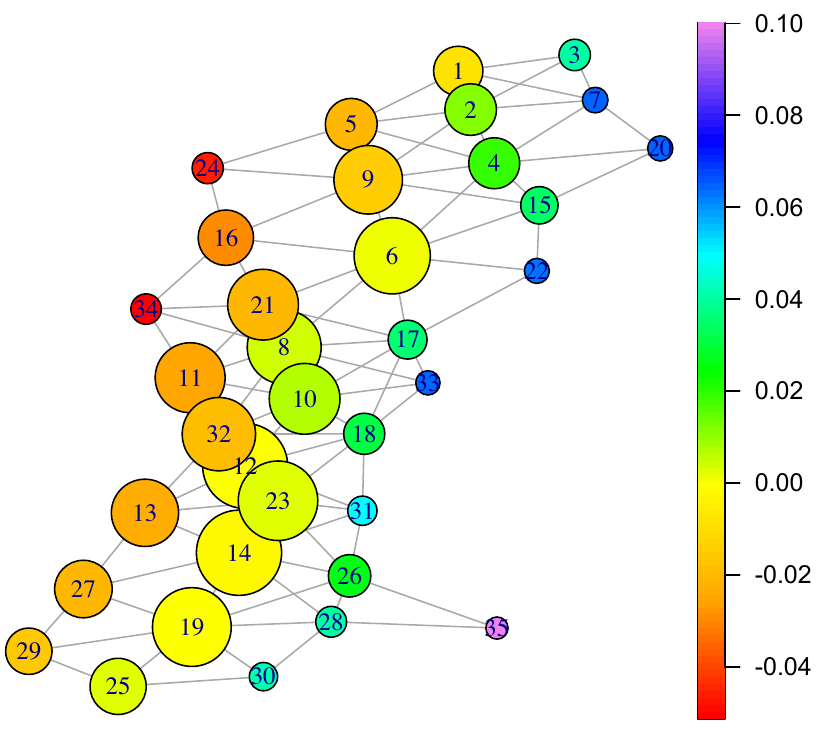}
         \caption{}
     \end{subfigure}
\caption{Clusters 9-12 in BMGraph of the Chromatic polynomial data shown in Figure \ref{fig:ch_9branch} at radius $\epsilon=0.06$ colored by the average PC1 score (A), and the average PC2 score (B) of graphs contained in each cluster.}
\label{fig:ch_9_zoom1_pca}
\end{figure}

\begin{figure}
\centering
    \begin{subfigure}[]{0.32\textwidth}
         \includegraphics[width=0.85\textwidth]{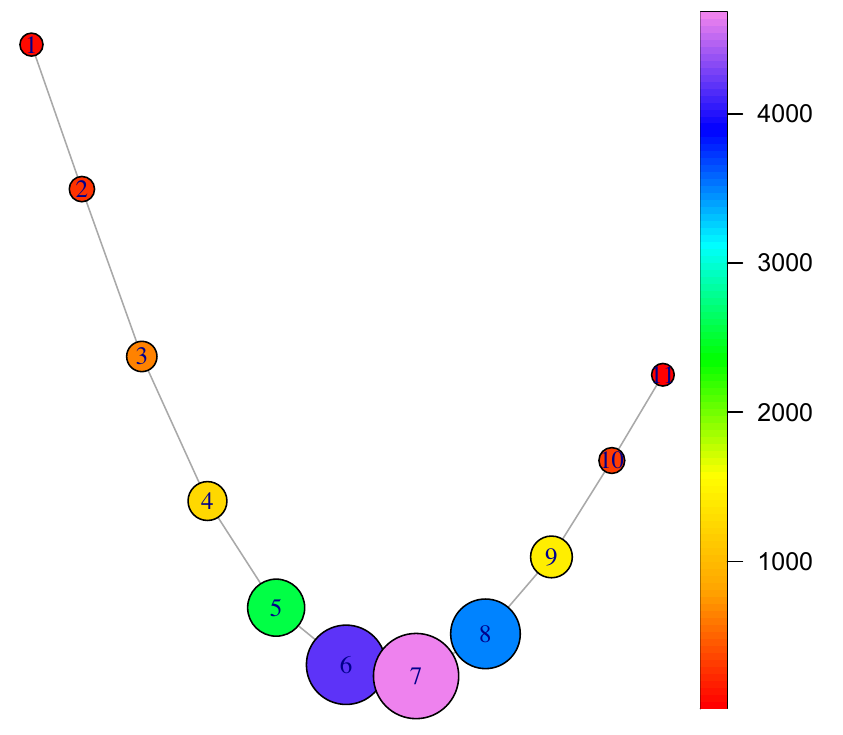}
         \caption{}
     \end{subfigure}
    \begin{subfigure}[]{0.32\textwidth}
         \includegraphics[width=0.85\textwidth]{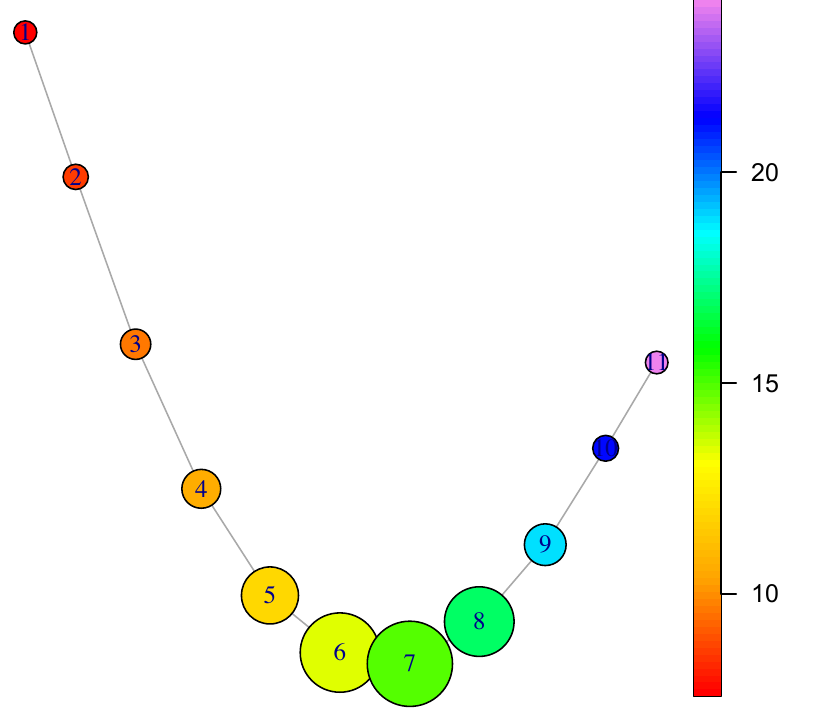}
         \caption{}
     \end{subfigure}
    \begin{subfigure}[]{0.32\textwidth}
         \includegraphics[width=0.85\textwidth]{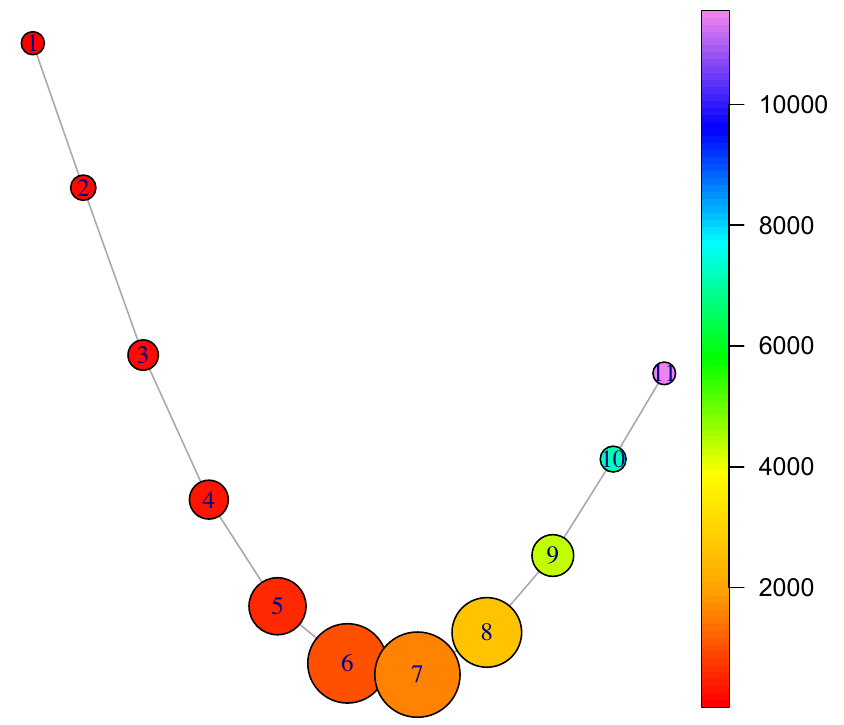}
         \caption{}
     \end{subfigure}
\caption{BMGraph of the chromatic polynomial data of all graphs with exactly 8 vertices and the Ball Mapper parameter  $\epsilon= 0.22$ colored by cluster size (A), number of edges (B), and the norm of the chromatic polynomial (C). }
\label{fig:ch_8branch}
\end{figure}

The linear direction within BMGraphs corresponds to the number of edges in a graph, see Figures  \ref{fig:ch_8branch}(B) and \ref{fig:ch_9branch}(B), as well as the $L^2$ norm of the chromatic coefficient vectors, see Figure \ref{fig:ch_8branch}(C), \ref{fig:ch_9branch}(C). This result can be explained by the fact that the number of edges $E$ contributes to the largest terms of each formula in Theorem \ref{Farrell}. 

\begin{figure}{h}
\centering
    \begin{subfigure}[]{0.32\textwidth}
         \includegraphics[width=0.85\textwidth]{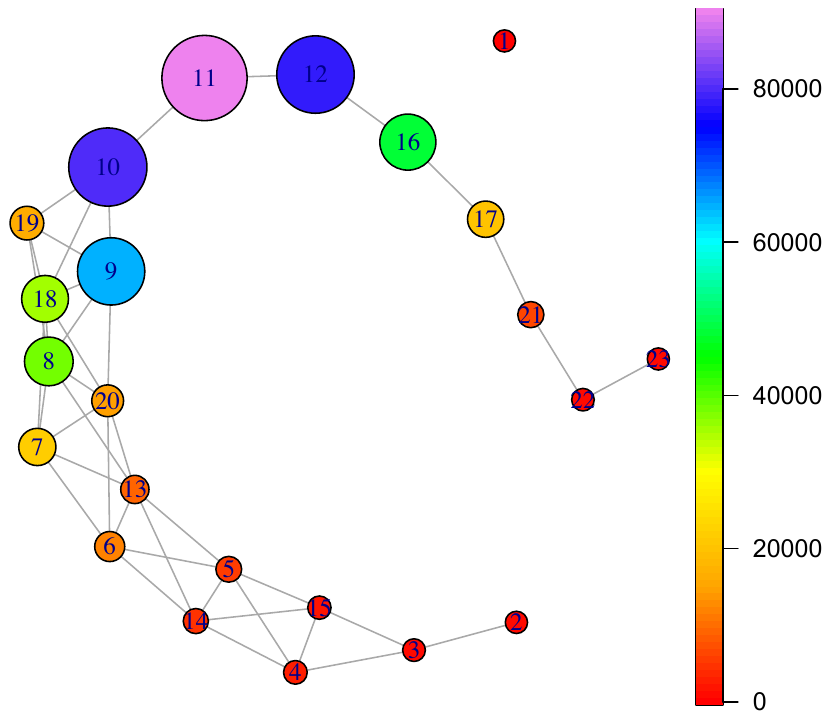}
         \caption{}
     \end{subfigure}
    \begin{subfigure}[]{0.32\textwidth}
         \includegraphics[width=0.85\textwidth]{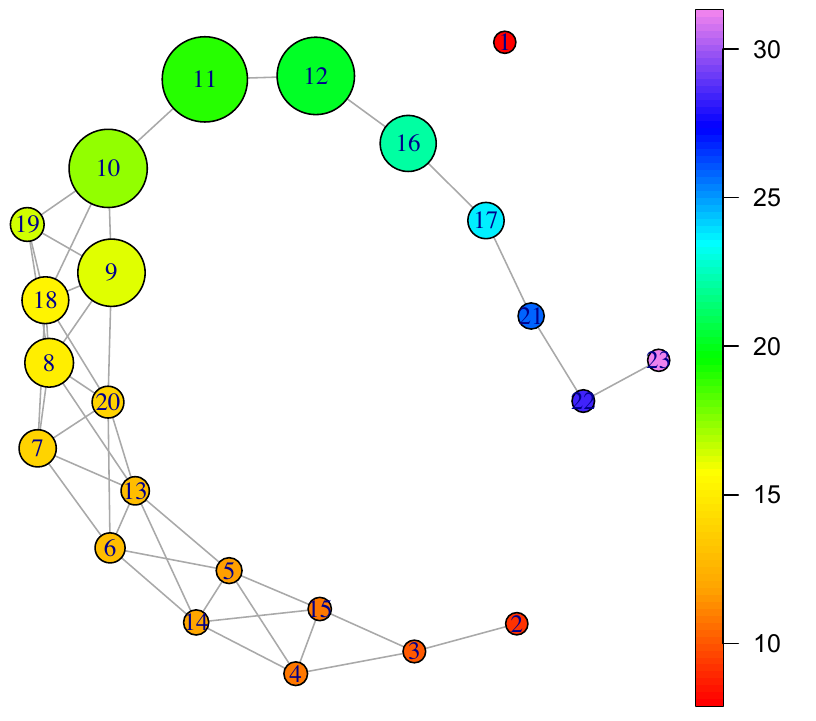}
         \caption{}
     \end{subfigure}
    \begin{subfigure}[]{0.32\textwidth}
         \includegraphics[width=0.85\textwidth]{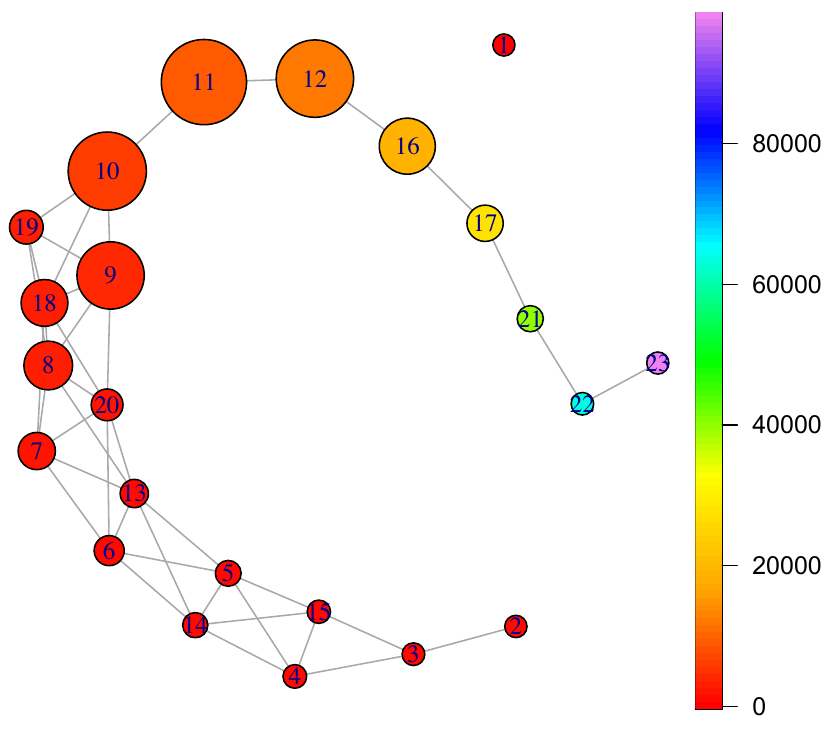}
         \caption{}
     \end{subfigure}
\caption{BMGraph of the chromatic polynomial data of all graphs with exactly 9 vertices and the Ball Mapper parameter  $\epsilon=0.15$ colored by cluster size (A), number of edges (B), and the norm of the chromatic polynomial (C). }
\label{fig:ch_9branch}
\end{figure}

Based on the BMGraph in Figure \ref{fig:ch_9branch}(A) and associated data, four largest clusters contain more than 50,000  out of total of 261079 points/graphs: cluster 9 (64442),  cluster 10 (80017), cluster 11 (90054),  and cluster 12 (78878). Therefore, we analyze the collection of these 4 clusters using a smaller value of radius $\epsilon$, see  BMGraphs on Figure \ref{fig:ch_9_zoom1}.  As expected,  the major PC1 direction of this particular embedding on the BMGraph  is  correlated with the number of edges  Figure \ref{fig:ch_9_zoom1}(A), and the PC2 direction is roughly consistent with the $L^2$ norm of the chromatic coefficients, Figure \ref{fig:ch_9_zoom1}(B). 
The average number of triangles in a graph contained in any given cluster, as well as the average degree distance seem to be varying in the PC2 direction, complementary to the main linear structure, from north-west to south-east in Figures  \ref{fig:ch_9_zoom1}(C-D). 

\begin{figure}
     \begin{subfigure}[]{0.23\textwidth}\centering
         \includegraphics[width=0.9\textwidth]{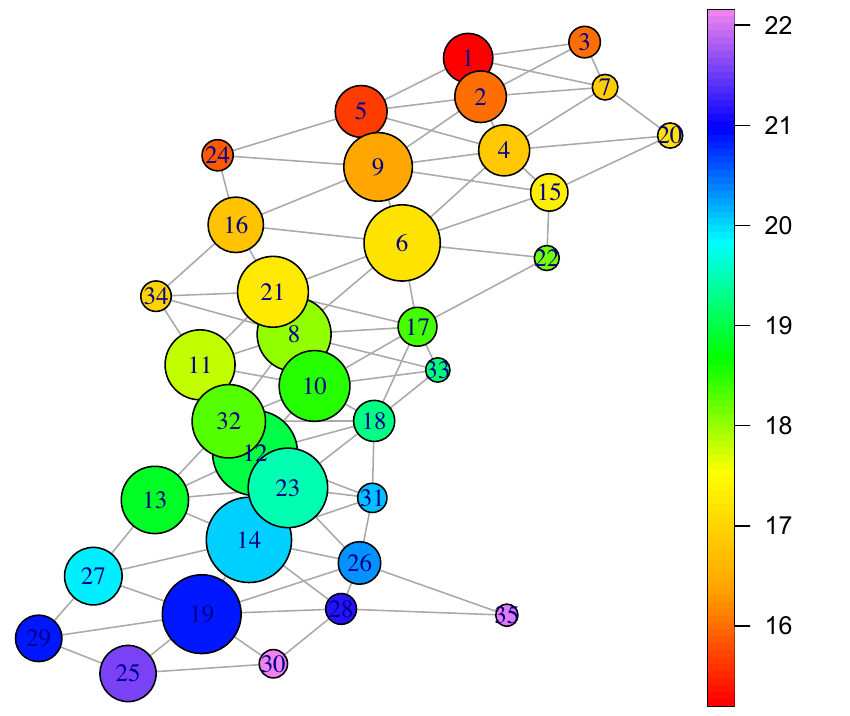}
         \caption{}
     \end{subfigure}
     \begin{subfigure}[]{0.23\textwidth}\centering
         \includegraphics[width=0.9\textwidth]{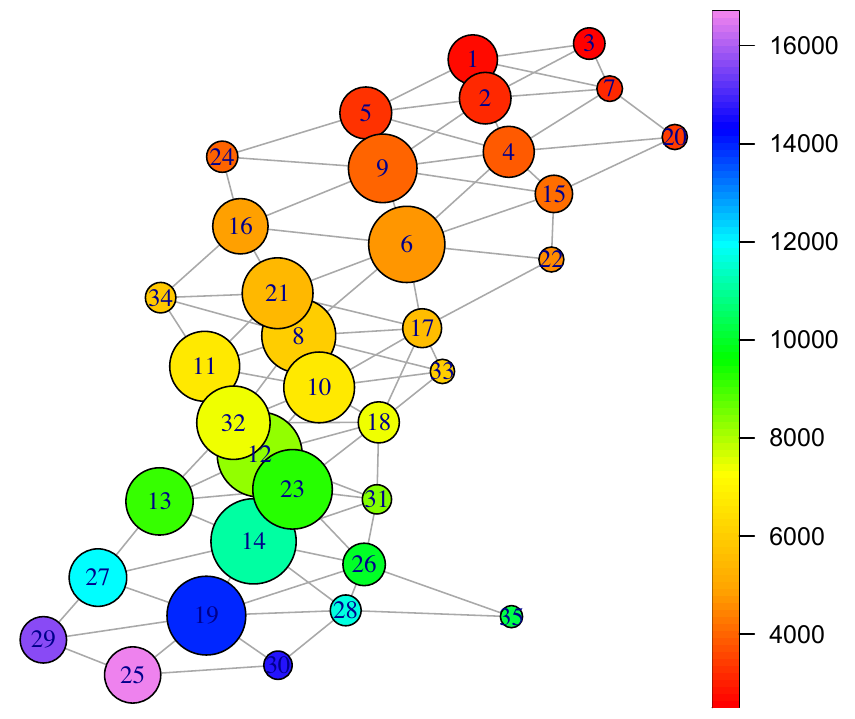}
         \caption{}
     \end{subfigure}
       \begin{subfigure}[]{0.23\textwidth}\centering
         \includegraphics[width=0.9\textwidth]{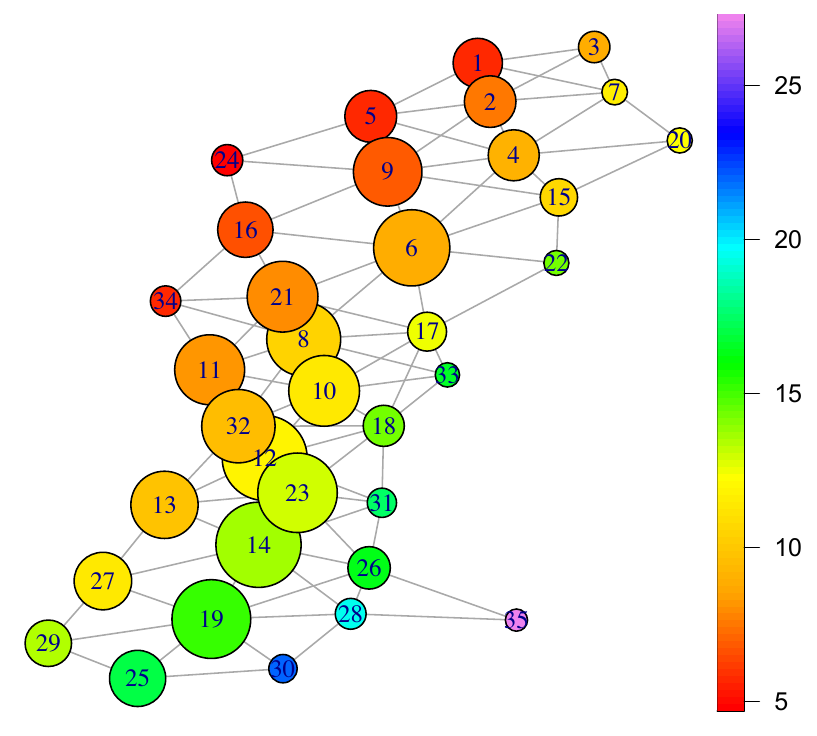}
         \caption{}
     \end{subfigure}
  \begin{subfigure}[]{0.23\textwidth}\centering
         \includegraphics[width=0.9\textwidth]{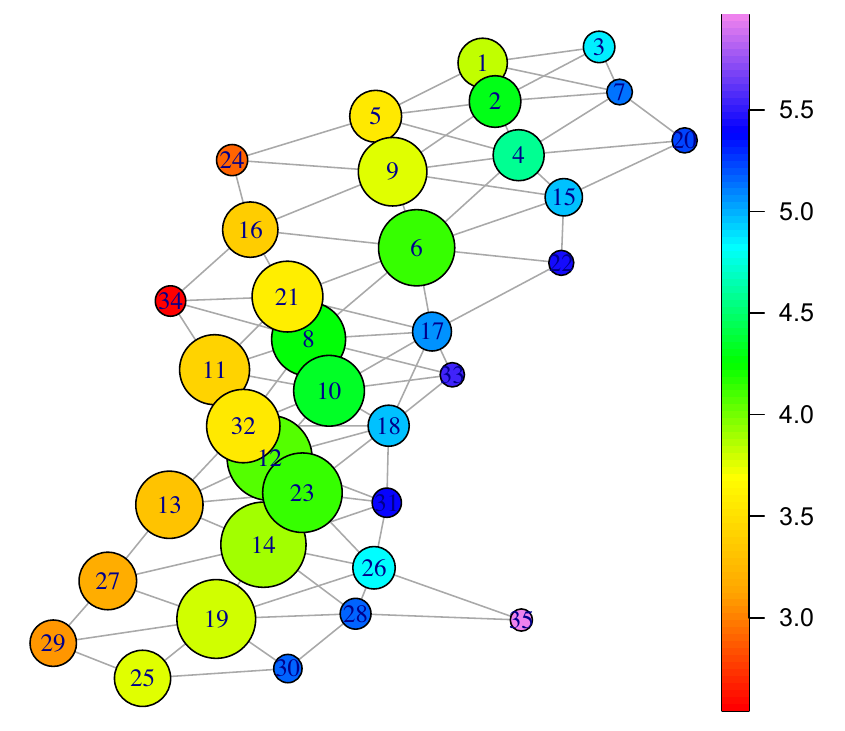}
         \caption{}
     \end{subfigure}
\caption{Clusters 9-12 in BMGraph of the chromatic polynomial data shown in  Figure \ref{fig:ch_9branch} at radius $\epsilon = 0.06$ colored by: 
number of edges (A) and  by $L^2$ norm of chromatic coefficients (B), average number of triangles in a graph (C), average difference between the maximal and minimal degree (D).}
\label{fig:ch_9_zoom1}
\end{figure} 


\subsection{Extremal chromatic polynomials and irregularity} \label{extrirr}
Section \ref{irreg_bkground} provides a few details on regular graphs and irregularity measures, but is by no means comprehensive. Note that there is a lot of ambiguity when it comes to irregularity measures. It was recently shown that no two of the following irregularity measures are mutually comparable: spectral irregularity, variance irregularity, Albertson index, and total irregularity  \cite{abdo2019graph}. 

 Section \ref{compress} offers a brief survey of results on graph compression, which uniformly increases or decreases values of a number of graph invariants related to the chromatic and Tutte polynomials. 

\begin{prop}
Suppose that $G'$ is obtained from $G$ by compression. Then
\begin{enumerate}
    \item The coefficients of the chromatic polynomial of $G'$ are smaller than those of  $G$: $|c_i(G')| \le |c_i(G)|$ for all $1 \le i \le n$. 
    \item Graph $G'$ has greater variance irregularity than $G$: $\sigma(G') \ge \sigma(G)$.
    \item Graph $G'$ has greater spectral irregularity than $G$: $\epsilon(G') \ge \epsilon(G)$.
\end{enumerate}
\end{prop}
\begin{proof}
    Part (1) is a restatement of Theorem 6.3 of  \cite{csikvari2011applications}. Part (2) derived from Lemma 2.3 of  \cite{cutler2011extremal} with Definition \ref{varirrdef} in mind. Part (3) follows from Theorem 2.1 of \cite{csikvari2009conjecture} with respect to  Definition \ref{spectralirrdef}.
\end{proof}

As a consequence, if we consider graphs with a fixed number of vertices and edges, then part (1) claims that threshold graphs and their co-chromatic graphs are among the chromatically minimal and part (2) says that threshold graphs maximize  variance irregularity. 

The rest of this section provides visualizations of these observations and a large-scale comparison between the spectral and variance irregularity measures, highlighting the relation between threshold graphs (the extremal graphs with respect to compression) and the coefficients of the chromatic polynomial.
For example, the graphs with extremal PCA2 scores and minimal coefficient graphs, as in Figure \ref{fig:chPCA9features1}(C), are either threshold or co-chromatic with threshold graphs, see Figure \ref{fig:chPCA9thresh}.

\begin{figure}
\centering
         \includegraphics[height=0.4\textwidth]{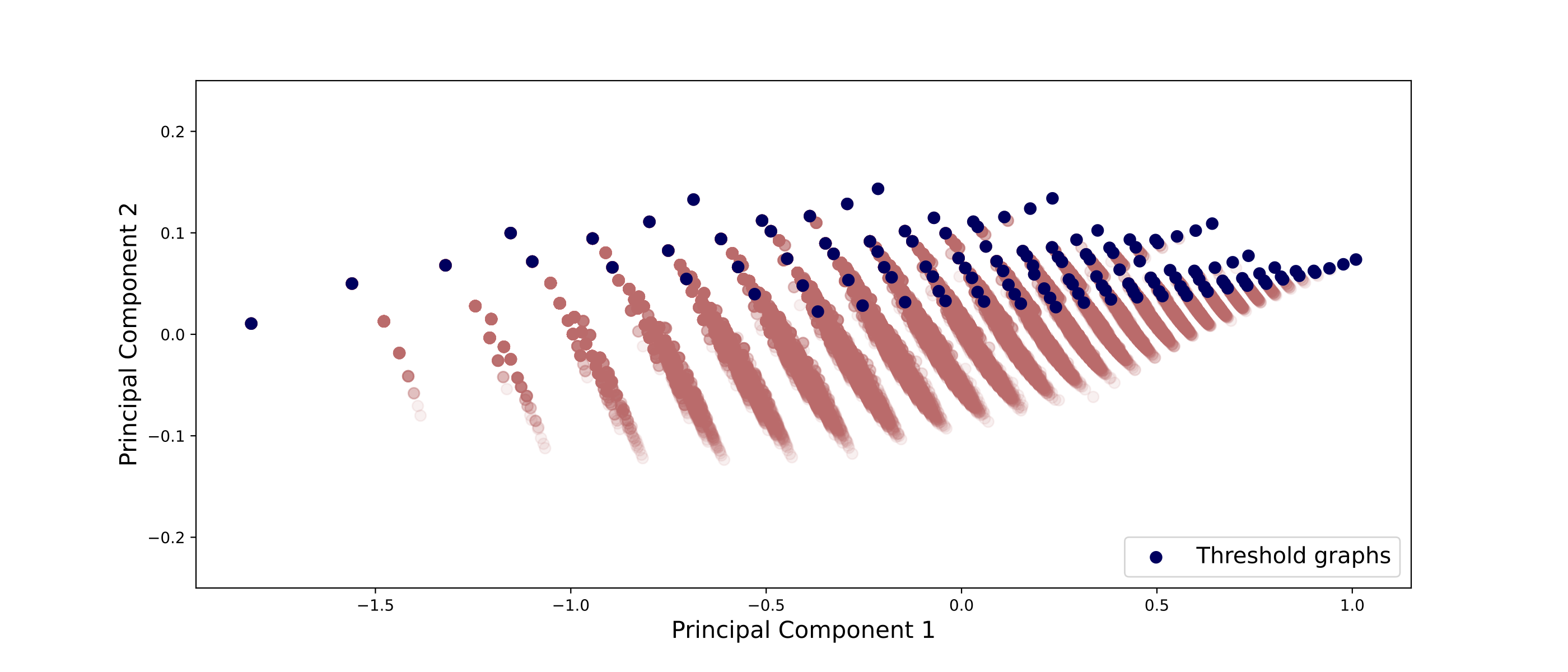}
\caption{A PCA projection of the chromatic polynomial data for all graphs of order 9 into 2-dimensions determined by the two most significant principal directions, PC1 and PC2, colored by presence of threshold graphs denoted by blue color. }
\label{fig:chPCA9thresh}
\end{figure}
The BMGraph \ref{fig:ch_9_zoom2}(A) is colored by spectral irregularity and (B) by variance irregularity emphasizing that, in these particular embeddings of  BMGraphs, clusters on the left- and right-hand sides contain graphs with extreme irregularity measures, with the average irregularity monotonically increasing from left to right.
The same phenomenon appears in Figure \ref{fig:ch_9_zoom1}(D) when the coloring function is the difference between maximal and minimal vertex degrees.

\begin{figure}
\centering
  \begin{subfigure}[]{0.23\textwidth}\centering
         \includegraphics[width=0.9\textwidth]{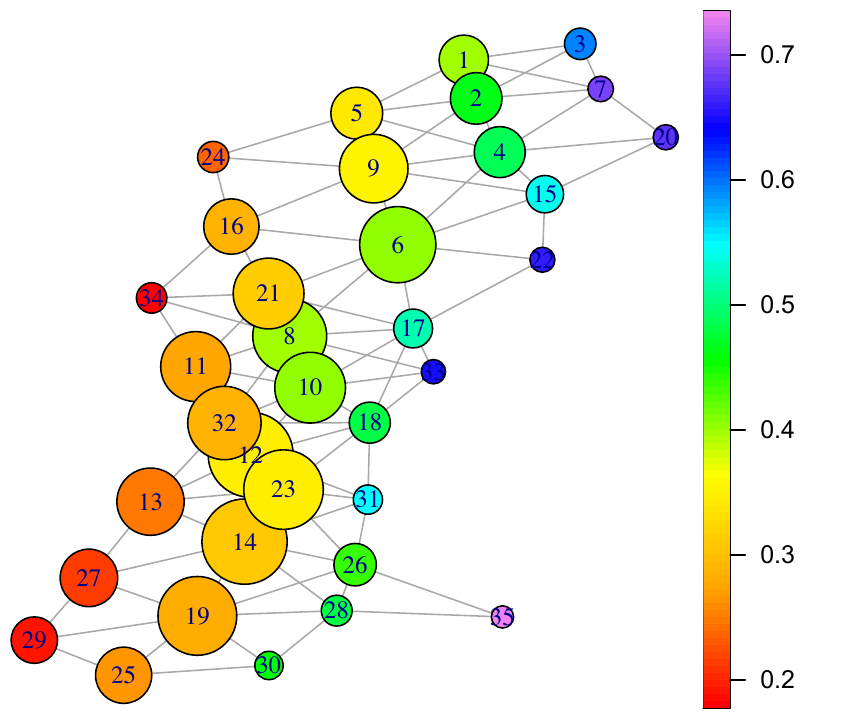}
         \caption{}
     \end{subfigure}
  \begin{subfigure}[]{0.23\textwidth}\centering
         \includegraphics[width=0.9\textwidth]{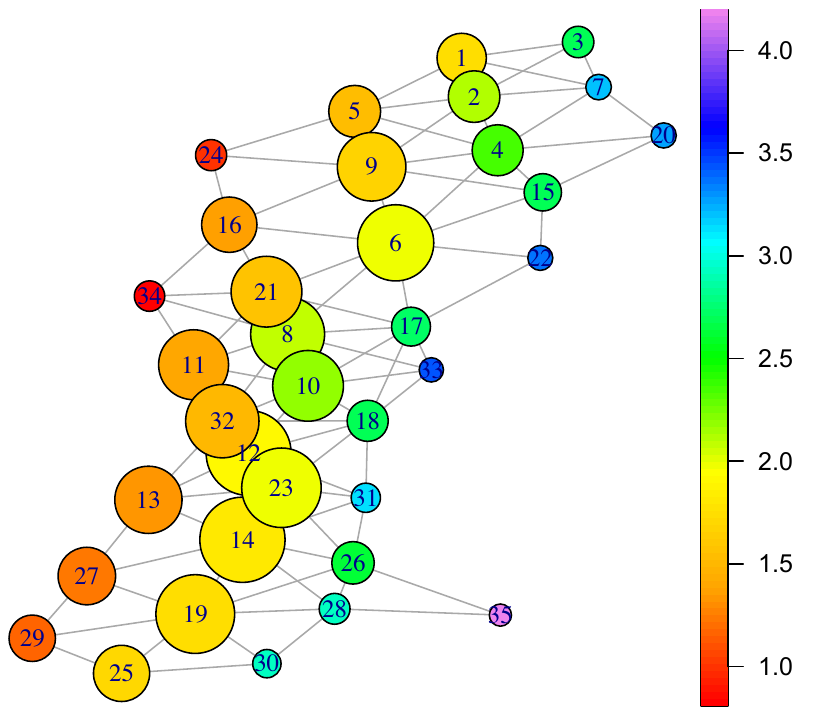}
         \caption{}
     \end{subfigure}
     \begin{subfigure}[]{0.23\textwidth}\centering
         \includegraphics[width=0.9\textwidth]{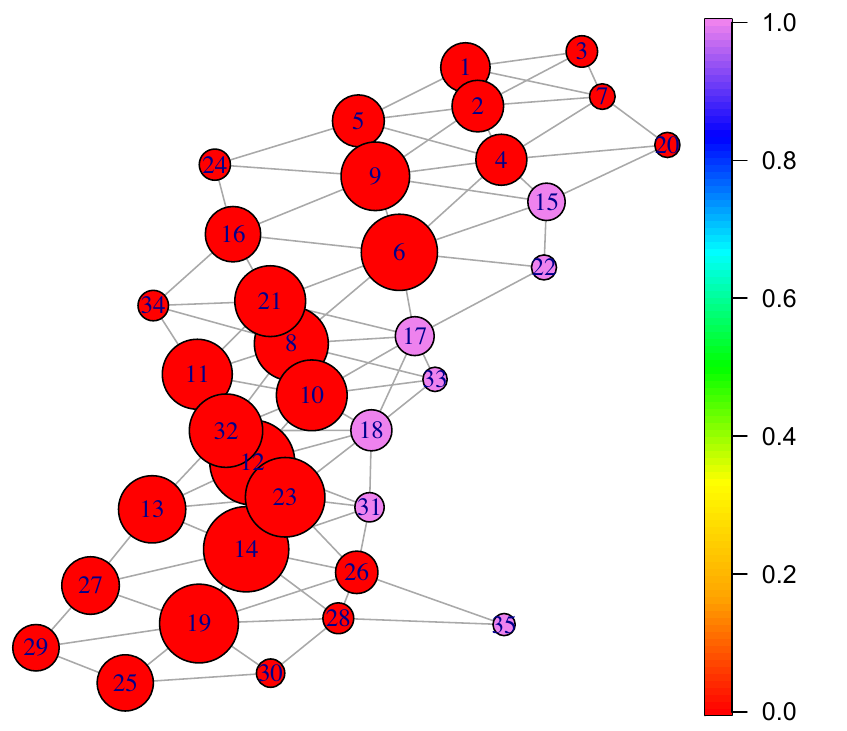}
         \caption{}
     \end{subfigure}
  \begin{subfigure}[]{0.23\textwidth}\centering
         \includegraphics[width=0.9\textwidth]{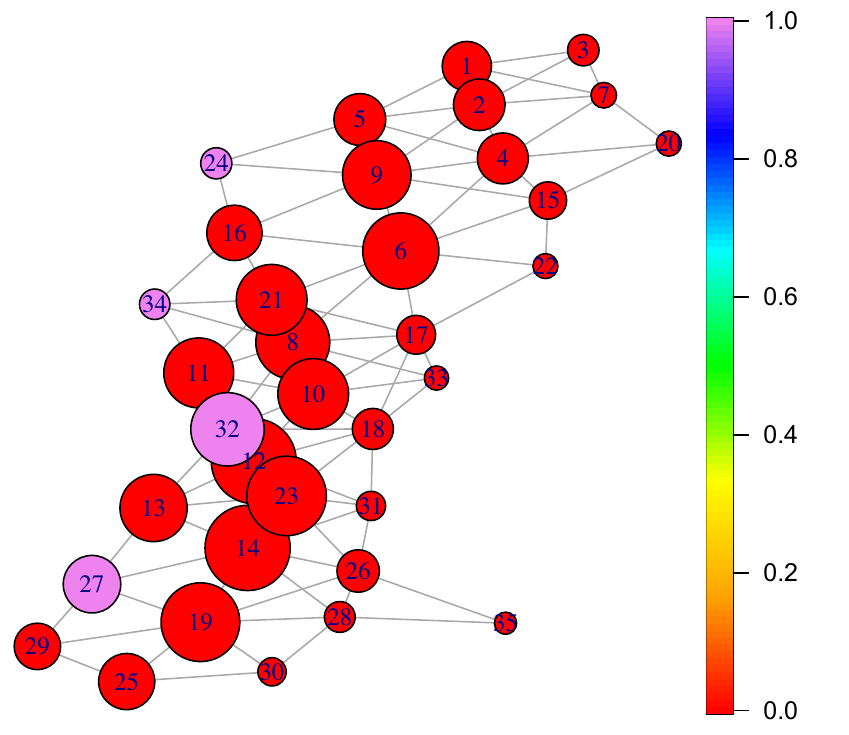}
         \caption{}
     \end{subfigure}
\caption{Clusters 9-12 in GMGraph of the Chromatic polynomial data shown in  Figure \ref{fig:ch_9branch} at radius $\epsilon = 0.06$ colored  by average measure of spectral irregularity in each cluster (A), variance irregularity (B), presence of chromatically minimal (purple color)   (C) and presence of chromatically maximal graphs (D)}.
\label{fig:ch_9_zoom2}
\end{figure}
These observations based on Ball Mapper graphs are matched by visualizations of the chromatic data  in the 2-dimensional PCA projection, see 
Figure \ref{fig:chPCA9features1}. In both projections, higher PC2 score corresponds to higher values of graph irregularity measures within the subset consisting of graphs with equal number of edges. However, the overall distributions of spectral Figures \ref{fig:chPCA9features1}(A), and variance irregularity Figures \ref{fig:chPCA9features1}(B) are different, with the variance regularity being lower on the middle left section. 

\begin{figure}
\centering
\begin{subfigure}[]{0.23\textwidth}\centering
         \includegraphics[height=0.8\textwidth]{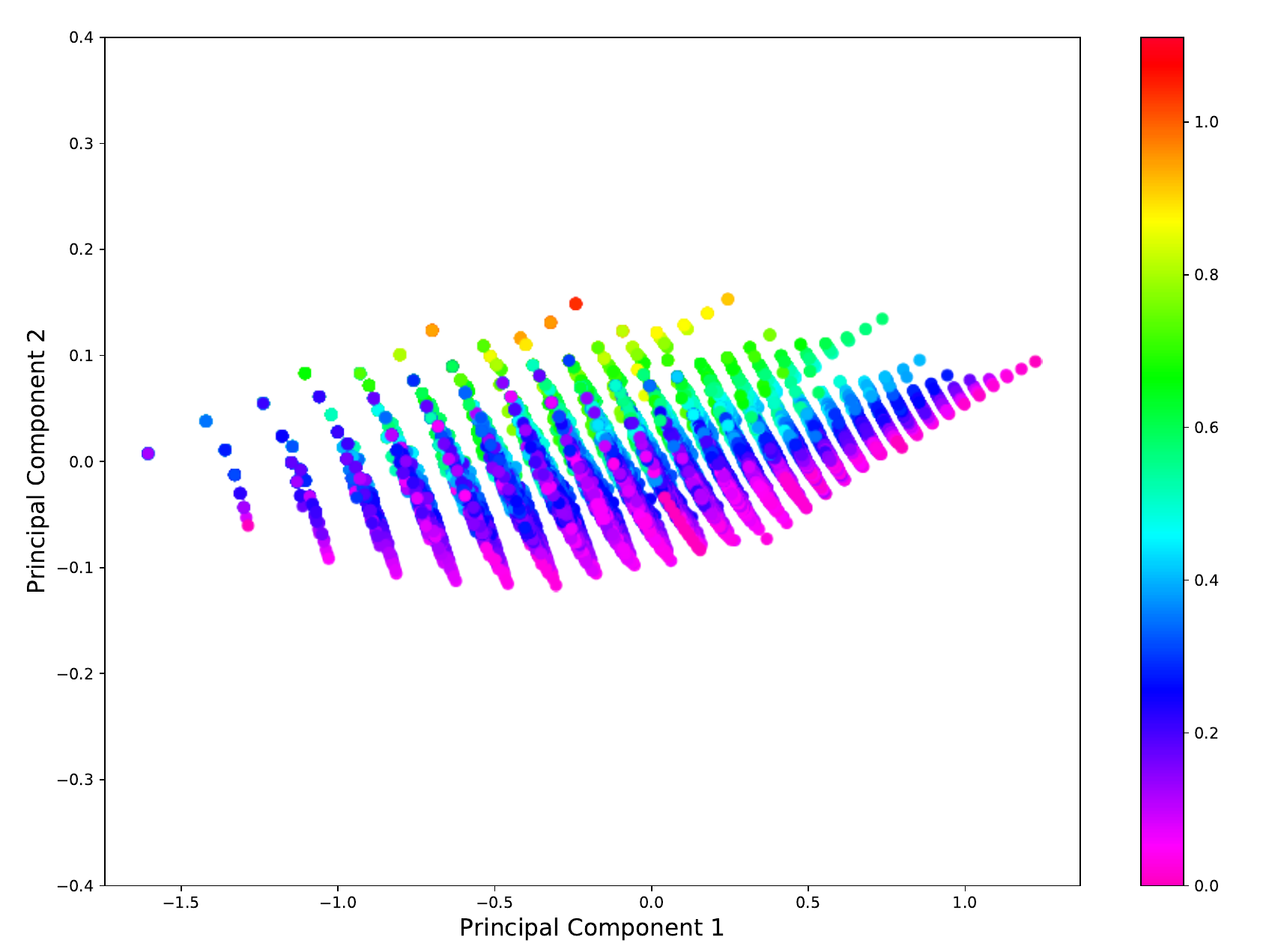}
         \caption{}
     \end{subfigure}
\begin{subfigure}[]{0.23\textwidth}\centering
         \includegraphics[height=0.8\textwidth]{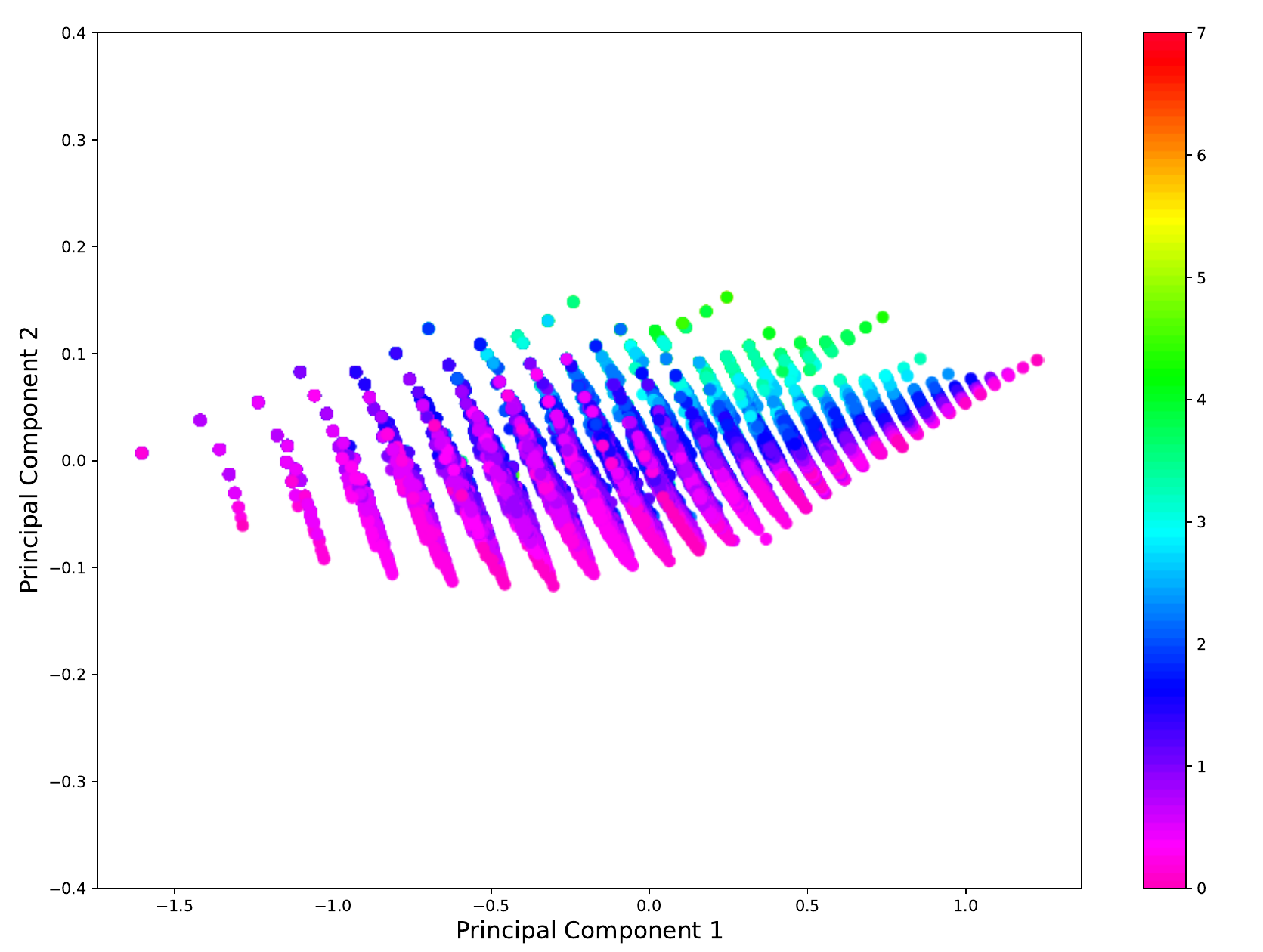}
         \caption{}
     \end{subfigure}
     \begin{subfigure}[]{0.23\textwidth}\centering
         \includegraphics[height=0.8\textwidth]{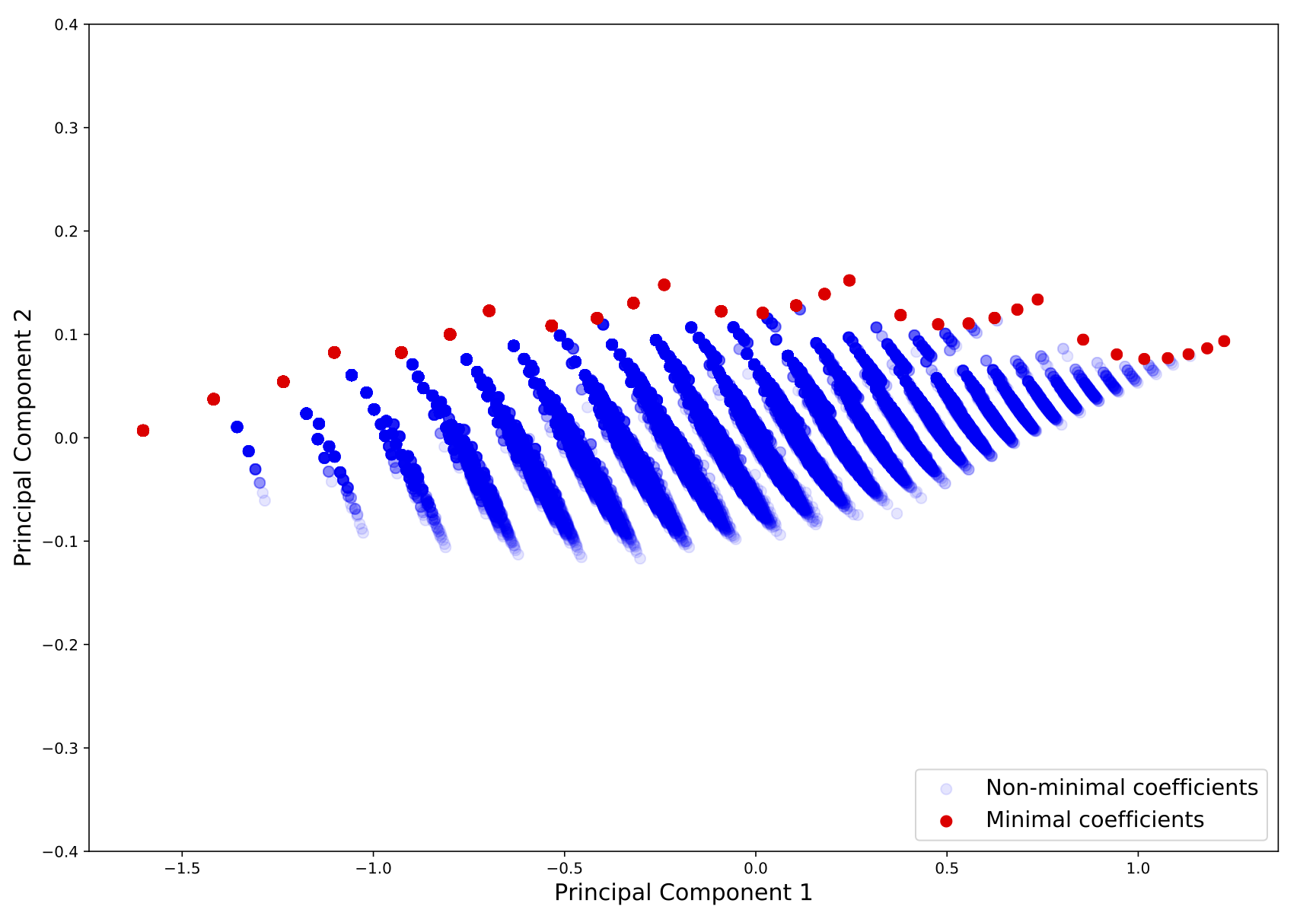}
         \caption{}
     \end{subfigure}
     \begin{subfigure}[]{0.23\textwidth}\centering
         \includegraphics[height=0.8\textwidth]{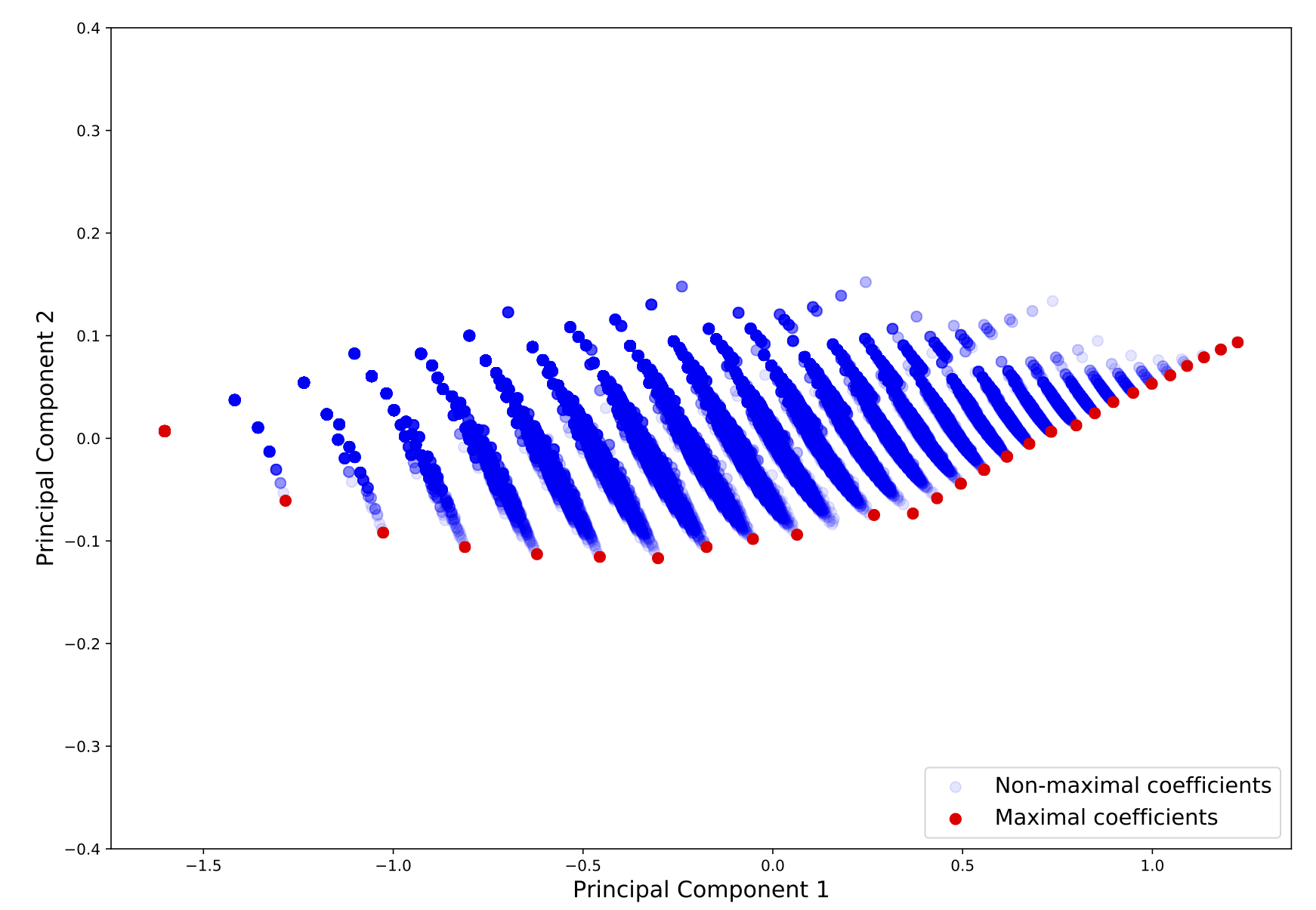}
         \caption{}
     \end{subfigure}
\caption{A PCA projection of the chromatic polynomial data for all graphs of order 9 into 2-dimensions determined by two most significant principal directions, PC1 and PC2, colored by spectral (A) and variance  irregularity measure (B), the presence of minimum-coefficient (C) and maximum-coefficient graphs (D) denoted by red color. }
\label{fig:chPCA9features1}
\end{figure}

The location of chromatically minimal and maximal graphs of order 9 (see Definition \ref{chromextremedef}) in the BMGraph and 2-dimensional PCA projection is shown on Figures \ref{fig:ch_9_zoom2}(C), \ref{fig:ch_9_zoom2}(D) and \ref{fig:chPCA9features1}(C), \ref{fig:chPCA9features1}(D), respectively. Both methods show that chromatically minimal graphs are found in regions of highest irregularity. This is expected since the families of minimal graphs characterized by Definition \ref{minfams} include members of the high-irregularity connected quasi-complete graphs and connected quasi-stars (Section \ref{irreg_bkground}). For example, it follows from Definition \ref{quasicomplete} that the connected quasi-complete graph $QC(n,m)$ is an element of the family $L(n,m)$.

While, to the best of our knowledge, there exists no similar characterization for chromatically maximal graphs, their location among the most regular graphs is supported by results from extremal graph theory. Lazebnik and others prove  partial characterizations for graphs which maximize evaluations of the  chromatic polynomial $P_G(\la)$ at given values of $\la$ \cite{lazebnik2019maximum}. These classes of graphs include Turan graphs (Definition \ref{turan}) and are generally similar to Turan graphs in structure. Turan graphs are extremal graphs which are chromatically unique \cite{chao1982maximally} and close to regular, see Section \ref{irreg_bkground}. The location of all Turan graphs of order 9 within the 2-dimensional PCA projection, among other chromatically maximal graphs, is highlighted in red in \ref{fig:chPCA9features1}(D).

Computations for graphs up to nine vertices imply that:  
\begin{itemize}
    \item The minimal and maximal elements of the chromatic polynomial poset $G(n,m)$ achieve the extremal PC2 scores among all graphs with $n$ vertices and $m$ edges when the principal components are computed over the data set of all $n$-vertex graphs.
\item The Turan graph $T(n,m)$ is maximal in the poset $G(n,m)$.
\end{itemize}

Further investigation of chromatically maximal graphs in our data suggests that 
there is a relationship between the ``tail" of the polynomial (the last few coefficients) and graph irregularity. 
 For example, consider the graphs in Figure \ref{fig:graphcompareex} with $n=9$ vertices and $E=11$ edges. The graph on the left is one of three co-chromatic non-isomorphic graphs which are minimal in $G(9,11)$, with higher irregularity measures and significantly smaller tail than the graph in Figure \ref{fig:graphcompareex}(B), which is maximal in $G(9,11).$

\begin{figure}
\centering
   \begin{subfigure}[]{0.45\textwidth}\centering
         \includegraphics[height=0.4\textwidth]{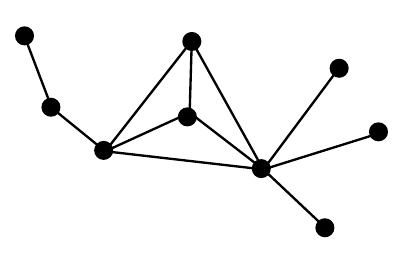}
         \caption{}
     \end{subfigure}
   \begin{subfigure}[]{0.45\textwidth}\centering
         \includegraphics[height=0.5\textwidth]{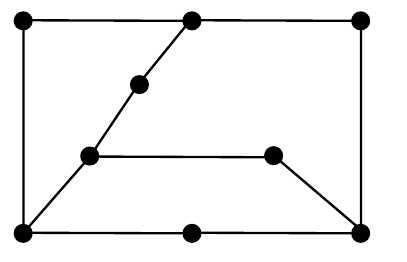}
         \caption{}
     \end{subfigure}
\caption{Extremal graphs in $G(9,11)$ with coefficient vectors 
equal to:  (1, 11, 51, 131, 205, 201, 121, 41, 6) (A), and (1, 11, 55, 165, 328, 446, 406, 224, 56) (B).} 
\label{fig:graphcompareex}
\end{figure}


\bibliographystyle{amsplain}
\bibliography{SazdanovicScofieldChromatic}{}
\end{document}